\documentclass[review,onefignum,onetabnum]{siamart171218}

\usepackage{enumerate}
\usepackage{amsmath}
\usepackage{mathrsfs}
\usepackage{xcolor}	
\usepackage{amssymb}
\usepackage{bm}
\usepackage{graphicx}
\usepackage{subcaption}
\usepackage{adjustbox}
\usepackage{soul}
\usepackage{listings}
\usepackage{array}  

\usepackage{amssymb,latexsym}
\usepackage{multirow}

\usepackage{color,epstopdf}
\usepackage{enumitem}
\usepackage{setspace}
\usepackage{url}
\usepackage{scalerel}
\usepackage[ruled,vlined,algo2e,linesnumbered]{algorithm2e}

\usepackage{footnote}
\makesavenoteenv{algorithm}
\usepackage{mdframed}
\newmdenv[linecolor=green!50!black, fontcolor=green!50!black, backgroundcolor=green!20, linewidth=2pt, roundcorner=10pt]{gnote}

\definecolor{amber}{rgb}{1.0, 0.49, 0.0}
\definecolor{darkpastelgreen}{rgb}{0.01, 0.75, 0.24}
\definecolor{darkred}{rgb}{0.64, 0.0, 0.0}
\definecolor{amberhl}{rgb}{1.0, 0.75, 0.0}

\newmdenv[linecolor=blue!50!black, fontcolor=blue!50!black, backgroundcolor=blue!20, linewidth=2pt, roundcorner=10pt]{anote}

\newcolumntype{L}{>{$}l<{$}}

\headers{NPASA - Local Convergence}{J. Diffenderfer, W. W. Hager}

\title{NPASA: An algorithm for nonlinear programming - Local Convergence\thanks{Submitted to the editors \textcolor{red}{ADD DATE}.\funding{The authors gratefully acknowledge support by the National Science Foundation under Grant 1819002, and by the Office of Naval Research under Grants N00014-15-1-2048 and N00014-18-1-2100. This work was performed under the auspices of the U.S. Department of Energy by Lawrence Livermore National Laboratory  under  Contract  DE-AC52-07NA27344, LLNL-JRNL-824570-DRAFT.}}}
\author{James Diffenderfer\thanks{Lawrence Livermore National Laboratory,
    Livermore, CA (\email{diffenderfer2@llnl.gov})}
  \and William W. Hager\thanks{The University of Florida,
    Gainesville, FL (\email{hager@ufl.edu}, \url{http://people.clas.ufl.edu/hager/})}}

\ifpdf
\hypersetup{
  pdftitle={NPASA: An algorithm for nonlinear programming - Local Convergence},
  pdfauthor={J. Diffenderfer, W. W. Hager}
}
\fi

\RequirePackage[normalem]{ulem} %DIF PREAMBLE
\RequirePackage{color}\definecolor{RED}{rgb}{1,0,0}\definecolor{BLUE}{rgb}{0,0,1} %DIF PREAMBLE
\begin{document}

\maketitle

% REQUIRED
\begin{abstract}
In this paper, we provide local convergence analysis for the two phase Nonlinear Polyhedral Active Set Algorithm (NPASA) designed to solve nonlinear programs. In particular, we establish local quadratic convergence of the primal iterates and global error estimator for NPASA under reasonable assumptions. Additionally, under the same set of assumptions we prove that only phase two of NPASA is executed after finitely many iterations. This paper is companion to a paper that provides motivation and global convergence analysis for NPASA \cite{diffenderfer2020global}.
\end{abstract}

% REQUIRED
\begin{keywords}
 nonlinear programming, global convergence, local convergence
\end{keywords}

% REQUIRED
\begin{AMS}
 90C30, 65K05 %(90C30: Operations research, mathematical programming; Nonlinear programming. 65K05: Numerical Analysis; Mathematical programming methods)
\end{AMS}

% Body
\section {Introduction}
\label{section:introduction}
In this paper, we establish local convergence properties for the Nonlinear Polyhedral Active Set Algorithm (NPASA) designed to solve nonlinear programming problems that was presented in a companion paper \cite{diffenderfer2020global}. In the analysis of NPASA, we will commonly refer to a general nonlinear program in the form
\begin{align}
\begin{array}{cc}
\displaystyle \min_{\bm{x} \in \mathbb{R}^{n}} & f(\bm{x}) \\
\text{s.t.} & \bm{h}(\bm{x}) = \bm{0}, \ \bm{r} (\bm{x}) \leq \bm{0}
\end{array} \label{prob:main-nlp}
\end{align}
where $f : \mathbb{R}^n \to \mathbb{R}$ and $\bm{h} : \mathbb{R}^n \to \mathbb{R}^{\ell}$ are nonlinear functions and $\bm{r} : \mathbb{R}^n \to \mathbb{R}^m$ is a linear function defined by
\begin{align}
\bm{r} (\bm{x}) := \bm{A} \bm{x} - \bm{b}, \label{def:r(x)}
\end{align}
where $\bm{A} \in \mathbb{R}^{m \times n}$ and $\bm{b} \in \mathbb{R}^m$. Additionally, note that we will often write $\Omega$ to denote the polyhedral constraint set given by $\Omega = \{ \bm{x} \in \mathbb{R}^n : \bm{r} (\bm{x}) \leq \bm{0} \}$ 

\iffalse
As established in \cite{diffenderfer2020global}, NPASA is split into two phases with branching criterion for switching between the phases. The first phase makes use of augmented Lagrangian techniques through use of the Global Step (GS) algorithm, Algorithm~\ref{alg:gs}, to ensure global convergence under few assumptions. The use of augmented Lagrangian techniques to ensure global convergence has been successful in several implementations of solvers for nonlinear programs \cite{Andreani2008, Conn1991, Conn1996, Gill2005, snopt77}. The second phase uses active set techniques by using the %Constraint and Multiplier Step (CMS) algorithm,
Local Step (LS) Algorithm~\ref{alg:cms}, to achieve fast local convergence. This approach differs from many of the popular techniques for achieving fast local convergence in iterative solvers for nonlinear programs \cite{NandW} like sequential quadratic programming \cite{Byrd2006, Fletcher1998, snopt77, GilSWdnopt} and interior point methods \cite{Wachter2005L, Wachter2006}. 
\fi

The focus of this paper is solely on establishing local convergence properties of NPASA. Global convergence properties of NPASA were established in a companion paper \cite{diffenderfer2020global}. Additionally, as thorough motivation for NPASA was presented in the companion paper, we omit a typical motivation and discussion of NPASA from this paper and refer the reader to Sections~1 and 3 from \cite{diffenderfer2020global} for these details. In order to facilitate the local convergence analysis and to maintain a sense of completeness in this paper, key details from \cite{diffenderfer2020global}, such as pseudocode and essential theorems, are included in Sections~\ref{sec:Error-Estimators} and \ref{section:npasa} of this paper. The paper is organized as follows. In Section~\ref{sec:Error-Estimators}, we briefly outline the error estimators used by NPASA for solving problem (\ref{prob:main-nlp}). In Section~\ref{section:npasa}, we provide pseudocode for the NPASA algorithm. %Sections~\ref{section:npasa-phase-one} and
Section~\ref{section:npasa-phase-two} focuses on establishing a local convergence result for phase two of NPASA and in Section~\ref{section:npasa-global-convergence} we prove that the primal iterates and global error estimator for NPASA are locally quadratically convergent. We now provide notation used throughout the paper. 

\subsection{Notation} \label{subsec:notation}
We will write $\mathbb{R}_+$ to denote the set $\{ x \in \mathbb{R} : x \geq 0 \}$. Scalar variables are denoted by italicized characters, such as $x$, while vector variables are denoted by boldface characters, such as $\bm{x}$. For a vector $\bm{x}$ we write $x_j$ to denote the $j$th component of $\bm{x}$ and we write $\bm{x}^{\intercal}$ to denote the transpose of $\bm{x}$. Outer iterations of NPASA will be enumerated by boldface characters with subscript $k$, such as $\bm{x}_k$, and we will write $x_{kj}$ to denote the $j$th component of the iterate $\bm{x}_k$. Inner iterations of NPASA will be enumerated by boldface characters with subscript $i$. If $\mathcal{S}$ is a subset of indices of $\bm{x}$ then we write $\bm{x}_{\mathcal{S}}$ to denote the subvector of $\bm{x}$ with indices $\mathcal{S}$. Additionally, for a matrix $\bm{M}$ we write $\bm{M}_{\mathcal{S}}$ to denote the submatrix of $\bm{M}$ with row indices $\mathcal{S}$. The ball with center $\bm{x}$ and radius $r$ will be denoted by $\mathcal{B}_r (\bm{x})$. Real valued functions are denoted by italicized characters, such as $f(\cdot)$, while vector valued functions are denoted by boldface characters, such as $\bm{h}(\cdot)$. The gradient of a real valued function $f(\bm{x})$ is denoted by $\nabla f (\bm{x})$ and is a row vector. The Jacobian matrix of a vector valued function $\bm{h} : \mathbb{R}^n \to \mathbb{R}^{\ell}$ is denoted by $\nabla \bm{h} \in \mathbb{R}^{\ell \times n}$. Given a vector $\bm{x}$, we will write $\mathcal{A} (\bm{x})$ denote the set of active constraints at $\bm{x}$, that is $\mathcal{A} (\bm{x}) := \{ i \in \mathbb{N} : \bm{a}_i^{\intercal} \bm{x} = 0 \}$ where $\bm{a}_i^{\intercal}$ is the $i$th row of matrix $\bm{A}$. For an integer $j$, we will write $f \in \mathcal{C}^j$ to denote that the function $f$ is $j$ times continuously differentiable. We write $c$ to denote a generic nonnegative constant that takes on different values in different inequalities. Given an interval $[a,b] \subset \mathbb{R}$, we write $Proj_{[a,b]} (\bm{x})$ to denote the euclidean projection of each component of the vector $\bm{x}$ onto the interval $[a,b]$. That is, if $\bm{v} = Proj_{[a,b]} (\bm{x})$ then the $j$th component of $\bm{v}$ is given by
\begin{align}
v_j = \left\{
	\begin{array}{lll}
	a & : & x_j \leq a \\
	x_j & : & a < x_j < b \\
    b & : & x_j \geq b
	\end{array}
	\right. .
\end{align}

In order to simplify the statement of several results throughout the paper, we provide abbreviations for several assumptions that are used in the convergence analysis. Note that each assumption will be clearly referenced when required.
\begin{enumerate}[leftmargin=2\parindent,align=left,labelwidth=\parindent,labelsep=7pt]
%\item[(\textbf{LI})] Linear Independence: Given $\bm{x}$, the matrix {\footnotesize $\displaystyle \begin{bmatrix} \nabla \bm{h} (\bm{x}) \\ \bm{A} \end{bmatrix}$} is of full row rank.
\item[(\textbf{LICQ})] Linear Independence Constraint Qualification: Given $\bm{x}$, {\footnotesize $ \begin{bmatrix} \nabla \bm{h} (\bm{x}) \\ \bm{A}_{\mathcal{A}(\bm{x})} \end{bmatrix}$} has full row rank.
\item[(\textbf{SCS})] Strict complementary slackness holds at a minimizer of problem (\ref{prob:main-nlp}). That is, Definition~\ref{def:scs} holds at a minimizer of problem (\ref{prob:main-nlp}).
\item[(\textbf{SOSC})] The second-order sufficient optimality conditions hold for a feasible point in problem (\ref{prob:main-nlp}). That is, the hypotheses of Theorem~\ref{thm:sosc} hold.
\item[(\textbf{SSOSC})] The strong second-order sufficient optimality conditions hold for a feasible point in problem (\ref{prob:main-nlp}). That is, the hypotheses of Theorem~\ref{thm:ssosc} hold.
\end{enumerate}
\section{Error Estimators for Nonlinear Optimization} \label{sec:Error-Estimators}
We now briefly recall the definitions of error estimators for problem (\ref{prob:main-nlp}) and any useful results relating to these error estimators. For more discussion and overview of these error estimators or proofs of these results please refer to Section~2 of \cite{diffenderfer2020global}. Both error estimators make use of the Lagrangian function $\mathcal{L}: \mathbb{R}^n \times \mathbb{R}^{\ell} \times \mathbb{R}^m \to \mathbb{R}$ for problem (\ref{prob:main-nlp}) defined by
\begin{align}
\mathcal{L} (\bm{x}, \bm{\lambda}, \bm{\mu}) = f(\bm{x}) + \bm{\lambda}^{\intercal} \bm{h}(\bm{x}) + \bm{\mu}^{\intercal} \bm{r}(\bm{x}).
\end{align}
The first error estimator is defined over the set $\mathcal{D}_0 := \{ (\bm{x}, \bm{\lambda}, \bm{\mu}) : \bm{x} \in \Omega, \bm{\lambda} \in \mathbb{R}^{\ell}, \bm{\mu} \geq \bm{0} \}$ and denoted $E_0 : \mathcal{D}_0 \to \mathbb{R}$ by 
\begin{align}
E_0 (\bm{x}, \bm{\lambda}, \bm{\mu}) = \sqrt{ \|\nabla_x \mathcal{L}(\bm{x}, \bm{\lambda}, \bm{\mu})\|^2 + \|\bm{h}(\bm{x})\|^2 - \bm{\mu}^{\intercal} \bm{r}(\bm{x})}. \label{def:E0}
\end{align}
The second estimator, denoted by $E_1$, makes use of the componentwise minimum function, denoted here by $\bm{\Phi} : \mathbb{R}^m \times \mathbb{R}^m \to \mathbb{R}^m$, so that the $i$th component is defined by
\begin{align}
\Phi_i \left(\bm{x}, \bm{y} \right) = \min\{ x_i, y_i\},
\end{align}
for $1 \leq i \leq m$. Then $E_1: \mathbb{R}^n \times \mathbb{R}^{\ell} \times \mathbb{R}^m \to \mathbb{R}$ is defined by 
\begin{align}
E_1 (\bm{x}, \bm{\lambda}, \bm{\mu}) = \sqrt{ \|\nabla_x \mathcal{L}(\bm{x}, \bm{\lambda}, \bm{\mu})\|^2 + \|\bm{h}(\bm{x})\|^2 + \|\bm{\Phi}\left(-\bm{r}(\bm{x}), \bm{\mu} \right)\|^2}. \label{def:E1}
\end{align}
The error bound properties for $E_1$ and $E_0$ are provided in Theorem~\ref{thm:E1-error-bound} and Corollary~\ref{cor:E0-error-bound}, respectively.

We define the \emph{multiplier} portion of the error estimators $E_0$ and $E_1$ by
\begin{align}
E_{m,0} (\bm{x}, \bm{\lambda}, \bm{\mu}) := \|\nabla_x \mathcal{L}(\bm{x}, \bm{\lambda}, \bm{\mu})\|^2 - \bm{\mu}^{\intercal} \bm{r}(\bm{x}) \label{def:Em0}
\end{align}
and
\begin{align}
E_{m,1} (\bm{x}, \bm{\lambda}, \bm{\mu}) := \|\nabla_x \mathcal{L}(\bm{x}, \bm{\lambda}, \bm{\mu})\|^2 + \|\bm{\Phi}\left(-\bm{r}(\bm{x}), \bm{\mu} \right)\|^2, \label{def:Em1}
\end{align}
respectively, and the \emph{constraint} portion of the error estimators by
\begin{align}
E_c (\bm{x}) := \| \bm{h} (\bm{x}) \|^2. \label{def:Ec}
\end{align}
Then from the definitions of $E_0$ and $E_1$ in (\ref{def:E0}) and (\ref{def:E1}) it follows that
\begin{align}
E_j (\bm{x}, \bm{\lambda}, \bm{\mu})^{2} = E_{m,j} (\bm{x}, \bm{\lambda}, \bm{\mu}) + E_c (\bm{x}), \label{eq:E=Em+Ec}
\end{align}
for $j \in \{0, 1\}$. 

We now state results pertaining to $E_0$ and $E_1$ that are required in the local convergence analysis of NPASA. Additional discussion of these results can be found in \cite{diffenderfer2020global}.
\begin{lemma} \label{lem:E0-E1-relationship}
Suppose $(\bm{x}, \bm{\lambda}, \bm{\mu}) \in \mathcal{D}_0 := \{ (\bm{x}, \bm{\lambda}, \bm{\mu}) : \bm{x} \in \Omega, \bm{\lambda} \in \mathbb{R}^{\ell}, \bm{\mu} \geq \bm{0} \}$. Then 
\begin{align}
E_{m,1} (\bm{x}, \bm{\lambda}, \bm{\mu}) \leq E_{m,0} (\bm{x}, \bm{\lambda}, \bm{\mu}) 
\qquad \text{and} \qquad
E_1 (\bm{x}, \bm{\lambda}, \bm{\mu}) \leq E_0 (\bm{x}, \bm{\lambda}, \bm{\mu}). \label{result:E0-E1-relationship}
\end{align}
\end{lemma}
%%%%%%%%%%%%%%%%%%%%%%%%%%%%%%%%%%%%%%%%

\begin{theorem} \label{thm:E1-error-bound}
Suppose that $\bm{x}^*$ is a local minimizer of problem (\ref{prob:main-nlp}) and that $f, \bm{h} \in \mathcal{C}^2$ at $\bm{x}^*$. If there exists $\bm{\lambda}^*$ and $\bm{\mu}^* \geq \bm{0}$ such that the KKT conditions and the second-order sufficient optimality conditions hold satisfied at $(\bm{x}^*, \bm{\lambda}^*, \bm{\mu}^*)$ then there exists a neighborhood $\mathcal{N}$ of $\bm{x}^*$ and a constant $c$ such that
\begin{align}
\|\bm{x} - \bm{x}^*\| + \|\bm{\lambda} - \hat{\bm{\lambda}}\| + \|\bm{\mu} - \hat{\bm{\mu}}\| 
\leq c E_1 (\bm{x}, \bm{\lambda}, \bm{\mu}) \label{result:E1-error-bound}
\end{align}
for all $\bm{x} \in \mathcal{N}$ where $(\hat{\bm{\lambda}}, \hat{\bm{\mu}})$ denotes the projection of $(\bm{\lambda}, \bm{\mu})$ onto the set of KKT multipliers at $\bm{x}^*$. 
\end{theorem}

\begin{corollary} \label{cor:E0-error-bound}
Suppose that the hypotheses of Theorem~\ref{thm:E1-error-bound} are satisfied at $(\bm{x}^*, \bm{\lambda}^*, \bm{\mu}^*)$. Then there exists a neighborhood $\mathcal{N}$ of $\bm{x}^*$ and a constant $c$ such that
\begin{align}
\|\bm{x} - \bm{x}^*\| + \|\bm{\lambda} - \hat{\bm{\lambda}}\| + \|\bm{\mu} - \hat{\bm{\mu}}\| 
\leq c E_0 (\bm{x}, \bm{\lambda}, \bm{\mu}) \label{result:E0-error-bound}
\end{align}
for all $(\bm{x}, \bm{\lambda}, \bm{\mu}) \in \mathcal{D}_0$ with $\bm{x} \in \mathcal{N}$ where $(\hat{\bm{\lambda}}, \hat{\bm{\mu}})$ denotes the projection of $(\bm{\lambda}, \bm{\mu})$ onto the set of KKT multipliers at $\bm{x}^*$. 
\end{corollary}

\begin{lemma} \label{lem:E1diff}
Let $(\bm{x}^*, \bm{\lambda}^*, \bm{\mu}^*)$ be a KKT point for problem (\ref{prob:main-nlp}) that satisfies assumption (\textbf{SCS}) and suppose that $f, \bm{h} \in \mathcal{C}^2$ in a neighborhood of $\bm{x}^*$. Then there exists a neighborhood $\mathcal{N}$ about $(\bm{x}^*, \bm{\lambda}^*, \bm{\mu}^*)$ such that $E_1 (\bm{x}, \bm{\lambda}, \bm{\mu}) \in \mathcal{C}^2$, for all $(\bm{x}, \bm{\lambda}, \bm{\mu}) \in \mathcal{N}$.
\end{lemma}

\section{NPASA: Algorithm for Nonlinear Constrained Optimization Problems} \label{section:npasa}
In this section, we provide pseudocode for NPASA designed to solve problem (\ref{prob:main-nlp}). For %We note that 
a detailed motivation and discussion of NPASA %is available in 
we direct the reader to Section~3 of the companion paper \cite{diffenderfer2020global}. NPASA is made up of two phases and each phase consists of a set of subproblems. To simplify the presentation of NPASA, we composed algorithms that make up the bulk of each phase. Phase one of NPASA contains the Global Step (GS) algorithm while phase two contains the %Constraint and Multiplier Step (CMS) algorithm. 
Local Step (LS) algorithm. Pseudocode for GS and LS can be found in Algorithm~\ref{alg:gs} and Algorithm~\ref{alg:cms}, respectively. Note that the %multiplier step in the CMS 
LS algorithm makes use of the following function
\begin{align}
\mathcal{L}_{p}^{i} (\bm{z}, \bm{\nu}) := f(\bm{z}) + \bm{\nu}^{\intercal} \bm{h}(\bm{z}) + p \| \bm{h}(\bm{z}) - \bm{h}(\bm{z}_i) \|^2 \label{def:penalized-Lagrangian}
\end{align}
and that we will write $(\bm{x}_k, \bm{\lambda}_k, \bm{\mu}_k)$ to denote the current primal-dual iterate of NPASA.

NPASA is then comprised of the GS and LS algorithms with criterion for branching between these algorithms. Pseudocode for NPASA can be found in Algorithm~\ref{alg:npasa}. The focus of the analysis in this paper is on phase two of NPASA which contains the LS algorithm. Phase one contains the GS algorithm and updates for the penalty parameter used in GS which were analyzed in Section~4 of \cite{diffenderfer2020global} to establish global convergence properties for NPASA. As such, there is no analysis of phase one of NPASA included in this paper as all necessary analysis was completed in \cite{diffenderfer2020global}. 

\begin{center}
\begin{algorithm}
%	\algsetup{linenosize=\notsotiny}
%	\notsotiny
	\caption{GS - Global Step Algorithm} \label{alg:gs} 
%	\begin{algorithmic}[1]
        % Initialize parameters
        {\bfseries Inputs:} Initial guess $(\bm{x}, \bm{\lambda}, \bm{\mu})$ and scalar parameters $\phi > 1$, $\bar{\lambda} > 0$, and $q$\\
        % Phase one
%		\STATE{}
		$\displaystyle \bm{\bar{\lambda}} = Proj_{[-\bar{\lambda}, \bar{\lambda}]} \left( \bm{\lambda} \right)$\\
		$\displaystyle \bm{x}' = \arg\min \left\{ \mathcal{L}_{q} \left( \bm{x}, \bm{\bar{\lambda}} \right) : \bm{x} \in \Omega \right\}$\\
		Set $\bm{\lambda}' \gets \bm{\bar{\lambda}} + 2 q \bm{h}(\bm{x}')$\\
		Construct $\bm{\mu}(\bm{x}', 1)$ from PPROJ output\footnote{Using formula (\ref{eq:pasa-construct-mult.13})} and set $\bm{\mu}' \gets \bm{\mu}(\bm{x}', 1)$\\
%    	\STATE{}
		{\bfseries Return} $(\bm{x}', \bm{\lambda}', \bm{\mu}')$
%	\end{algorithmic}
\end{algorithm}
\end{center}

\begin{center}
\begin{algorithm}
%	\algsetup{linenosize=\notsotiny}
%	\notsotiny
	\caption{LS - Local Step Algorithm} \label{alg:cms} 
        % Initialize parameters
		{\bfseries Inputs:} Initial guess $(\bm{x}, \bm{\lambda}, \bm{\mu})$ and scalar parameters $\theta \in (0,1)$, $\alpha \in (0, 1]$, $\beta \geq 1$, $\sigma \in (0,1)$, $\tau \in (0,1)$, $p >> 1$, $\delta \in (0,1)$, $\gamma \in (0,1)$ \\ %, $i_C \in \mathbb{Z}_+$, $i_M \in \mathbb{Z}_+$\\
        % Phase two 
        \emph{Constraint Step}: Set $\bm{w}_0 \gets \bm{x}$ and set $i \gets 0$.\\
%        \For{$i = 0, 1, \ldots, i_C - 1$}{
        \While{$E_c (\bm{w}_i) > \theta E_{m,1} (\bm{x}, \bm{\lambda}, \bm{\mu})$}{
            Choose $p_i$ such that $p_i \geq \max \left\{ \beta^2, \| \bm{h} (\bm{w}_i) \|^{-2} \right\}$ and set $s_i = 1$\\
  			$\left[ \bm{\overline{w}}_{i+1}, \bm{y}_{i+1} \right] = \arg\min \left\{\| \bm{w} - \bm{w}_i \|^2 + p_i \| \bm{y} \|^2 : \nabla \bm{h} (\bm{w}_i) (\bm{w} - \bm{w}_i) + \bm{y} = - \bm{h} (\bm{w}_i), \bm{w} \in \Omega \right\}$\\
            Set $\alpha_{i+1} \gets 1 - \| \bm{y}_{i+1} \|$\\
            \eIf{$\alpha_{i+1} < \alpha$}{
                \textbf{Return} $(\bm{x}, \bm{\lambda}, \bm{\mu})$
            }{
            \While{$\| \bm{h} (\bm{w}_i + s_i (\bm{\overline{w}}_{i+1} - \bm{w}_{i})) \| > (1 - \tau \alpha_{i+1} s_i) \| \bm{h} (\bm{w}_i) \|$}{
                Set $s_i \gets \sigma s_i$
            }
    	    Set $\bm{w}_{i+1} \gets \bm{w}_i + s_i ( \bm{\overline{w}}_{i+1} - \bm{w}_i)$ and set $i \gets i+1$
            }
        }
   	    Set $\bm{w} \gets \bm{w}_i$\\
        \emph{Multiplier Step}: Set $\bm{z}_0 \gets \bm{w}$, set $p_0 \gets p$, and set $i \gets 0$\\
    	$(\bm{\nu}_0, \bm{\eta}_0) \in \arg \min \left\{ E_{m,0} ( \bm{z}_0, \bm{\nu}, \bm{\eta}) + \gamma \| [ \bm{\nu}, \bm{\eta} ] \|^2 : \bm{\eta} \geq \bm{0} \right\}$\\
    	$\bm{\eta}_0' \in \arg\min \left\{ E_{m,1} ( \bm{z}_0, \bm{\nu}_0, \bm{\eta}) : \bm{\eta} \geq \bm{0} \right\}$\\
        %Set $K_0 \gets E_{m,0} ( \bm{z}_0, \bm{\nu}_0, \bm{\eta}_0)$\\
%        \For{$i = 0, 1, \ldots, i_M - 1$}{
        \While{$E_{m,1} (\bm{z}_i, \bm{\nu}_i, \bm{\eta}_i') > \theta E_c (\bm{w})$}{
            Increase $p_i$ if necessary\\
            $\bm{z}_{i+1} = \arg \min \left\{ \mathcal{L}_{p_i}^i (\bm{z}, \bm{\nu}_i) : \nabla \bm{h} (\bm{z}_i) (\bm{z} - \bm{z}_i) = 0, \bm{z} \in \Omega \right\}$\\
    	    $(\bm{\nu}_{i+1}, \bm{\eta}_{i+1}) \in \arg\min \left\{ E_{m,0} ( \bm{z}_{i+1}, \bm{\nu}, \bm{\eta}) + \gamma \| [ \bm{\nu}, \bm{\eta} ] \|^2 : \bm{\eta} \geq \bm{0} \right\}$\\
            $\bm{\eta}_{i+1}' \in \arg\min \left\{ E_{m,1} ( \bm{z}_{i+1}, \bm{\nu}_{i+1}, \bm{\eta}) : \bm{\eta} \geq \bm{0} \right\}$\\
            %\eIf{$E_{m,1} (\bm{z}_{i+1}, \bm{\nu}_{i+1}, \bm{\eta}_{i+1}') \leq \delta K_i$}{
            %    Set $K_{i+1} \gets E_{m,1} (\bm{z}_{i+1}, \bm{\nu}_{i+1}, \bm{\eta}_{i+1}')$
            %}{
            %    Set $(\bm{\nu}_{i+1}, \bm{\eta}_{i+1}') \gets (\bm{\nu}_i, \bm{\eta}_i')$ and $K_{i+1} \gets K_i$
            %}
            \eIf{$E_{m,1} (\bm{z}_{i+1}, \bm{\nu}_{i+1}, \bm{\eta}_{i+1}') > \delta E_{m,1} (\bm{z}_i, \bm{\nu}_i, \bm{\eta}_i')$}{
                \textbf{Return} $(\bm{x}, \bm{\lambda}, \bm{\mu})$
            }{
                Set $i \gets i+1$
            }
            
        }
		{\bfseries Return} $(\bm{z}_i, \bm{\nu}_i, \bm{\eta}_i')$
\end{algorithm}
\end{center}

\begin{center}
\begin{algorithm}
%	\algsetup{linenosize=\notsotiny}
%	\notsotiny
	\caption{NPASA - Nonlinear Polyhedral Active Set Algorithm} \label{alg:npasa} 
    % Initialize parameters
	{\bfseries Inputs:} Initial guess $(\bm{x}_{0}, \bm{\lambda}_{0}, \bm{\mu}_{0})$ and scalar parameters $\varepsilon > 0$, $\theta \in (0,1)$, $\phi > 1$, $\bar{\lambda} > 0$, $q_0 \geq 1$, $\alpha \in (0, 1]$, $\beta \geq 1$, $\sigma \in (0,1)$, $\tau \in (0,1)$, $p >> 1$, $\delta \in (0,1)$, $\gamma > 0$. \\%, $i_C \in \mathbb{Z}_+$, $i_M \in \mathbb{Z}_+$.\\
	Set $e_0 = E_1 (\bm{x}_{0}, \bm{\lambda}_{0}, \bm{\mu}_{0})$, $k = 0$, and goto phase one.\\
    % Phase one
	\textbf{------ Phase one ------}\\
	Set $q_k \gets \max \left\{ \phi, (e_{k-1})^{-1} \right\}  q_{k-1}$.\\
	\While{$E_{m,1} (\bm{x}_k, \bm{\lambda}_k, \bm{\mu}(\bm{x}_k, 1)) > \varepsilon$}{
	    \emph{Global Step}: $\displaystyle \left( \bm{x}_{k+1}, \bm{\lambda}_{k+1}, \bm{\mu}_{k+1} \right) \gets GS \left( \bm{x}_k, \bm{\lambda}_k, \bm{\mu}_k; \phi, \bar{\lambda}, q_k \right)$\\
		\emph{Update parameters}: $q_{k+1} \gets q_k$, $e_{k+1} \gets \min \left\{ E_1 (\bm{x}_{k+1}, \bm{\lambda}_{k+1}, \bm{\mu}_{k+1}), e_k \right\}$, $k \gets k+1$.\\
		\emph{Check branching criterion}:\\
        \If{$E_{m,1} (\bm{x}_k, \bm{\lambda}_k, \bm{\mu}_k) \leq \theta E_c (\bm{x}_{k-1})$}{
		    \textbf{goto} phase two
        }
    }
    % Phase two 
	\textbf{------ Phase two ------}\\
	\While{$E_{m,1} (\bm{x}_k, \bm{\lambda}_k, \bm{\mu}_k) > \varepsilon$}{
	    \emph{Local Step}: $\displaystyle \left( \bm{z}, \bm{\nu}, \bm{\eta} \right) \gets LS \left( \bm{x}_k, \bm{\lambda}_k, \bm{\mu}_k; \theta, \alpha, \beta, \sigma, \tau, p, \delta, \gamma \right)$\\
		\emph{Check branching criterion}:\\
        \eIf{$E_1 (\bm{z}, \bm{\nu}, \bm{\eta}) > \theta E_1 (\bm{x}_k, \bm{\lambda}_k, \bm{\mu}_k)$}{
           	\textbf{goto} phase one
        }{
            Set $(\bm{x}_{k+1}, \bm{\lambda}_{k+1}, \bm{\mu}_{k+1}) \gets (\bm{z}, \bm{\nu}, \bm{\eta})$ and set $e_{k+1} \gets \min \left\{ E_1 (\bm{x}_{k+1}, \bm{\lambda}_{k+1}, \bm{\mu}_{k+1}), e_k \right\}$\\
            Set $k \gets k+1$.
        }
    }
    {\bfseries Return} $(\bm{x}_k, \bm{\lambda}_k, \bm{\mu}_k)$
\end{algorithm}
\end{center}

\section{Convergence Analysis for Phase Two of NPASA} \label{section:npasa-phase-two}
In this section, we provide conditions under which phase two of NPASA satisfies desirable local convergences properties. From the pseodocode in Algorithms~\ref{alg:cms} and \ref{alg:npasa}, observe that the LS algorithm in phase two of NPASA is decomposed into constraint and multiplier steps. The constraint step is designed to move the primal iterates to a point satisfying the nonlinear constraints, $\bm{h} (\bm{x}) = \bm{0}$, while the multiplier step provides updates for the dual iterates while attempting to maintain primal iterates satisfying the nonlinear constraints. First, we prove fast local convergence properties of the constraint step in Section~\ref{subsec:constraint-step-analysis}. Then we provide a local convergence analysis of the multiplier step in Section~\ref{subsec:multiplier-step-analysis}. Lastly, in Section~\ref{subsec:npasa-phase-two-local-conv} we use our analysis of the constraint and multiplier steps to provide a local convergence result for phase two of NPASA. %which contains the LS algorithm. 
%%%%%%%%%%%%%%%%%%%%%%%%%%%%%%%%%%%%%%%%%%%%%%%%%%%%%%%%%%%%%%%%%%%%%%%%%%%%%%%

For several results in this section we will require some subset of the assumptions (\textbf{LICQ}), (\textbf{SCS}), (\textbf{SOSC}), and (\textbf{SSOSC}) provided in Section~\ref{section:introduction}. We will clearly specify which assumptions are used in the statement of each convergence result. Additionally, note that for the purposes of simplifying the analysis in Theorem~\ref{NewtonYSchemeThm}, we will express the polyhedral constraint set for problem (\ref{prob:main-nlp}), originally defined by $\Omega = \{ \bm{x} : \bm{r}(\bm{x}) \leq \bm{0} \}$, in an equivalent form $\Omega = \{ \bm{x} : \bm{g} (\bm{x}) = \bm{0} \}$, for some function $\bm{g} (\bm{x})$. Note that this equivalent formulation of $\Omega$ can be constructed by adding a slack variable to the inequality constraints in problem (\ref{prob:main-nlp}) and performing a change of variables to define $\bm{g} (\cdot)$, as illustrated in Section 3.3.2 of \cite{bertsekas1995nonlinear}. Further, using this approach we have that $\nabla \bm{g}_{\mathcal{S}}(\bm{x})$ is of full row rank provided that $\bm{A}_{\mathcal{S}}$ is of full row rank, for any $\mathcal{S} \subseteq \{1, 2, \ldots, m\}$. It should be noted that the choice to represent $\Omega$ in this equivalent form is solely to simplify the convergence analysis and that an implementation of NPASA can still use the definition of the set $\Omega$ as originally defined in Section~\ref{section:introduction}.% so that problem %(\ref{NewtonYScheme}) 
\subsection{Convergence Analysis for Constraint Step} \label{subsec:constraint-step-analysis}
In this section, we establish a local quadratic convergence rate of the multiplier error estimator $E_c$. Recall the iterative scheme for determining $\bm{w} \in \Omega$ such that $\bm{h} (\bm{w}) = \bm{0}$ used in the constraint step of the LS algorithm for a given initial point $\bm{w}_0$:
\begin{align}
(\bm{w}_{i+1}, \bm{y}_{i+1}) &= \arg\min \left\{\| \bm{w} - \bm{w}_i \|^2 + \| \bm{y} \|^2 : \nabla \bm{h} (\bm{w}_i) (\bm{w} - \bm{w}_i) + \frac{1}{\sqrt{p_i}} \bm{y} = - \bm{h} (\bm{w}_i), \bm{w} \in \Omega \right\} \label{NewtonYScheme} \\
\bm{w}_{i+1} &\gets \bm{w}_i + s_i ( \bm{w}_{i+1} - \bm{w}_i) \label{NewtonYScheme.0}
\end{align}
where $p_i \geq 1$ is a penalty parameter and where $s_i \leq 1$ is chosen such that
\begin{align}
\| \bm{h} (\bm{w}_i + s_i (\bm{w}_{i+1} - \bm{w}_{i})) \| \leq (1 - \tau \alpha_i s_i) \| \bm{h} (\bm{w}_i) \|. \label{NewtonArmijo}
\end{align} 
Here, $\alpha_i$ is defined by $\alpha_i := 1 - \| \bm{y}_i \|$ and to continue performing the scheme we require that $\alpha_i \in [\alpha, 1]$ for some fixed parameter $\alpha > 0$. Note that this is the iterative scheme in the constraint step of the LS algorithm, Algorithm~\ref{alg:cms}, where we have made the change of variables $\bm{y} \gets \frac{1}{\sqrt{p_i}} \bm{y}$ to simplify the analysis in Theorem~\ref{NewtonYSchemeThm}. By introducing the slack variable, $\bm{y}$, into our update scheme, it is no longer necessary to consider the solvability of this problem as the constraint set is always nonempty. 
%TODO
%However, Theorem~\ref{NewtonYSchemeThm} serves to update the convergence analysis performed in Theorem 2.1 of \cite{Hager1993} to account for the presence of the slack variable $\bm{y}$ and penalty parameter $p_i$ in (\ref{NewtonYScheme}). 

\begin{theorem} \label{NewtonYSchemeThm}
Suppose $\bm{x}^*$ satisfies $\bm{h} (\bm{x}^*) = \bm{0}$, $\bm{g} (\bm{x}^*) = \bm{0}$, $\bm{g}, \bm{h} \in \mathcal{C}^2$ near $\bm{x}^*$, and that assumption (\textbf{LICQ}) holds at $\bm{x}^*$. Then there exists a neighborhood $\mathcal{N}$ about $\bm{x}^*$ and a constant $c \in \mathbb{R}$ such that for any $\bm{w}_0 \in \mathcal{N} \cap \Omega$ and sequence of penalty parameters $\{ p_i \}_{i = 0}^{\infty}$ satisfying $p_i \geq \max \left\{ \beta^2, \| \bm{h} (\bm{w}_i) \|^{-2} \right\}$, with $\beta \geq 1$, the iterative scheme (\ref{NewtonYScheme}) -- (\ref{NewtonYScheme.0}) converges to a point $\bm{w}^*$ with the following properties: 
\iffalse
\begin{itemize}
\item[(\ref{NewtonYSchemeThm}.0)] $\bm{h}(\bm{w}^*) = \bm{g}(\bm{w}^*) = \bm{0}$ and $s_i = 1$ satisfies the Armijo line search given by (\ref{NewtonArmijo}) for each $i$.
\end{itemize}
\fi
\begin{align}
\bm{h}(\bm{w}^*) = \bm{0}, \bm{g}(\bm{w}^*) = \bm{0}, \ \text{and} \ s_i = 1 \ \text{satisfies Armijo line search given by (\ref{NewtonArmijo}) for each $i$}. \label{NewtonYSchemeThm.0}
\end{align}
Moreover, for all $i \geq 0$:
\iffalse
\begin{itemize}
\item[(\ref{NewtonYSchemeThm}.1)] $\| \bm{w}_i - \bm{w}_0 \| \leq c \| \bm{h}(\bm{w}_0) \|$
\item[(\ref{NewtonYSchemeThm}.2)] $\| \bm{w}_i - \bm{w}^* \| \leq c \| \bm{h}(\bm{w}_0) \| 2^{-2^i}$
\item[(\ref{NewtonYSchemeThm}.3)] $\| \bm{h}(\bm{w}_{i+1}) \| \leq c \| \bm{h}(\bm{w}_i) \|^2$.
\end{itemize}
\fi
\begin{align}
\| \bm{w}_i - \bm{w}_0 \| &\leq c \| \bm{h}(\bm{w}_0) \|, \label{NewtonYSchemeThm.1} \\
\| \bm{w}_i - \bm{w}^* \| &\leq c \| \bm{h}(\bm{w}_0) \| 2^{-2^i}, \label{NewtonYSchemeThm.2} \\
\| \bm{h}(\bm{w}_{i+1}) \| &\leq c \| \bm{h}(\bm{w}_i) \|^2. \label{NewtonYSchemeThm.3}
\end{align}
\end{theorem} 

\begin{proof}
Let $\mathcal{S} \subseteq [m] := \{ 1, 2, \ldots, m \}$ and let $\mathcal{A} := \mathcal{A} (\bm{x}^*)$, for simplicity. Define $\bm{F}_{p, \mathcal{S}}: \mathbb{R}^{2 \ell + |\mathcal{S}|} \times \mathbb{R}^{n} \to \mathbb{R}^{2 \ell + |\mathcal{S}|}$ by
\begin{align}
\bm{F}_{p,\mathcal{S}} (\bm{u}, \bm{y}, \bm{x}) = \begin{bmatrix}
\nabla \bm{h} (\bm{x}) \bm{M}_{\mathcal{S}}(\bm{x}) \bm{u} + \frac{1}{\sqrt{p}} \bm{y} + \bm{h} (\bm{x}) \\
\bm{g}_{\mathcal{S}} \left( \bm{x} + \bm{M}_{\mathcal{S}}(\bm{x}) \bm{u} \right) \\
\bm{y}
\end{bmatrix}
\end{align}
where $\bm{u} \in \mathbb{R}^{\ell + |\mathcal{S}|}$, $\bm{y} \in \mathbb{R}^{\ell}$, $\bm{M}_{\mathcal{S}}(\bm{x}) = \left[ \nabla \bm{h} (\bm{x})^{\intercal} \ \vert \ \nabla \bm{g}_{\mathcal{S}} (\bm{x})^{\intercal} \right] \in \mathbb{R}^{n \times (\ell + |\mathcal{S}|)}$, and $p \geq 1$ is a scalar. Taking $\mathcal{S} = [m]$, we claim that there exists neighborhoods $\mathcal{N}_1$ of $[\bm{0}, \bm{0}] \in \mathbb{R}^{\ell + m} \times \mathbb{R}^{\ell}$ and $\mathcal{N}_2$ of $\bm{x}^* \in \mathbb{R}^n$ such that the equation 
\begin{align}
\bm{F}_{p,[m]} (\bm{u}, \bm{y}, \bm{x}) = \bm{0} \label{reg-gen-eq.1}
\end{align}
has a unique solution $[\bm{u}, \bm{y}] = [\bm{u}_p (\bm{x}), \bm{y}_p (\bm{x})]$ in $\mathcal{N}_1$ for every $\bm{x} \in \mathcal{N}_2$. To prove this claim, we will make use of Theorem~2.1 in \cite{Robinson1980}. It follows from Corollary~3.2 in \cite{Robinson1980} that if $\nabla_{[u, y]} \bm{F}_{p,\mathcal{A}} (\bm{0}, \bm{0}, \bm{x}^*)$ is nonsingular and the Lipschitz constant of $\nabla_{[u, y]} \bm{F}_{p,\mathcal{A}} (\bm{0}, \bm{0}, \bm{x}^*)^{-1}$ is uniformly bounded for sufficiently large values of $p$ then the hypothesis of Theorem~2.1 in \cite{Robinson1980} is satisfied for equation (\ref{reg-gen-eq.1}). Evaluating the Jacobian of $\bm{F}_{p,\mathcal{A}}$ with respect to $[\bm{u}, \bm{y}]$ at $\bm{u} = \bm{0}$, $\bm{y} = \bm{0}$, and $\bm{x} = \bm{x}^*$ yields
\begin{align}
\nabla_{[u, y]} \bm{F}_{p,\mathcal{A}} (\bm{0}, \bm{0}, \bm{x}^*) 
&= \begin{bmatrix}
\nabla_{u} \bm{F}_{p,\mathcal{A}} (\bm{0}, \bm{0}, \bm{x}^*) \ \vert \ \nabla_{y} \bm{F}_{p,\mathcal{A}} (\bm{0}, \bm{0}, \bm{x}^*)
\end{bmatrix}
= \left[
\begin{array}{c|c}
\multirow{2}{*}{$\bm{M}_{\mathcal{A}}(\bm{x}^*)^{\intercal} \bm{M}_{\mathcal{A}}(\bm{x}^*)$} & \frac{1}{\sqrt{p}} \bm{I}_{\ell} \\
 & \bm{0} \\
\hline
\bm{0} & \bm{I}_{\ell}
\end{array}
\right]
\end{align}
which is nonsingular by the (\textbf{LICQ}) hypothesis and choice of $\bm{M}_{\mathcal{A}}(\bm{x})$. Note that $\bm{I}_{\ell}$ denotes the identity matrix of dimension $\ell \times \ell$. It remains to shows that the Lipschitz constant of $\nabla_{[u, y]} \bm{F}_{p,\mathcal{A}} (\bm{0}, \bm{0}, \bm{x}^*)^{-1}$ is uniformly bounded for sufficiently large values of $p$. For simplicity, define $\bm{B}_p := \nabla_{[u, y]} \bm{F}_{p,\mathcal{A}} (\bm{0}, \bm{0}, \bm{x}^*)^{-1}$, $\bm{C} := \left( \bm{M}_{\mathcal{A}}(\bm{x}^*)^{\intercal} \bm{M}_{\mathcal{A}}(\bm{x}^*)\right)^{-1}$, and 
\begin{align}
\bm{D} := \bm{C} \begin{bmatrix}
\bm{I}_{\ell} \\
\bm{0}
\end{bmatrix}
\end{align} 
and observe that 
\begin{align}
\bm{B}_p = \left[
\begin{array}{c|c}
\bm{C} & \frac{1}{\sqrt{p}} \bm{D} \\
\hline
\bm{0} & \bm{I}_{\ell}
\end{array}
\right].
\end{align}
As the desired Lipschitz constant is given by $\| \bm{B}_p \|$, it suffices to determine the singular values of $\bm{B}_p$. Since 
\begin{align}
\bm{B}_p \bm{B}_p^{\intercal} &= \left[
\begin{array}{c|c}
\bm{C} & \frac{1}{\sqrt{p}} \bm{D} \\
\hline
\bm{0} & \bm{I}_{\ell}
\end{array}
\right] \left[
\begin{array}{c|c}
\bm{C}^{\intercal} & \bm{0} \\
\hline 
\frac{1}{\sqrt{p}} \bm{D}^{\intercal} & \bm{I}_{\ell}
\end{array}
\right]
= \left[
\begin{array}{c|c}
\bm{C} \bm{C}^{\intercal} + \frac{1}{p} \bm{D} \bm{D}^{\intercal} & \frac{1}{\sqrt{p}} \bm{D} \\
\hline 
\frac{1}{\sqrt{p}} \bm{D}^{\intercal} & \bm{I}_{\ell}
\end{array}
\right]
\end{align}
we have that
\begin{align}
\lim_{p \to \infty} \bm{B}_p \bm{B}_p^{\intercal} = \left[
\begin{array}{c|c}
\bm{C} \bm{C}^{\intercal} & \bm{0} \\
\hline 
\bm{0} & \bm{I}_{\ell}
\end{array}
\right]. \label{NewtonYThm.3}
\end{align}
As the matrix in (\ref{NewtonYThm.3}) is a block diagonal matrix, its eigenvalues are the eigenvalues of $\bm{C} \bm{C}^{\intercal}$ and $\bm{I}_{\ell}$. As the nonzero singular values of $\bm{B}_p$ are the square roots of the nonzero eigenvalues of $\bm{B}_p \bm{B}_p^{\intercal}$, we conclude that the singular values of $\bm{B}_p$ are uniformly bounded for sufficiently large $p$. Hence, the hypothesis of Theorem~2.1 in \cite{Robinson1980} holds for equation (\ref{reg-gen-eq.1}). Thus, there exist neighborhoods $\mathcal{N}_1$ of $[\bm{0}, \bm{0}]$ and $\mathcal{N}_2$ of $\bm{x}^*$ such that equation (\ref{reg-gen-eq.1}) has a unique solution $[\bm{u}, \bm{y}] = [\bm{u}_p (\bm{x}), \bm{y}_p (\bm{x})]$ in $\mathcal{N}_1$ for every $\bm{x} \in \mathcal{N}_2$. Further, Theorem 2.1 of \cite{Robinson1980} yields that %this solution satisfies 
$\| [ \bm{u}_p (\bm{x}), \ \bm{y}_p (\bm{x}) ] \| 
\leq \sigma \| \bm{F}_{p,[m]} (\bm{0}, \bm{0}, \bm{x}) - \bm{F}_{p,[m]} (\bm{0}, \bm{0}, \bm{x}^*) \|$, for all $\bm{x} \in \mathcal{N}_2$ and sufficiently large $p$, where $\sigma$ is independent of $\bm{x}$ and $p$. 
It follows from the definition of $\bm{F}_{p,[m]}(\bm{u}, \bm{y}, \bm{x})$ and equation (\ref{reg-gen-eq.1}) that $\bm{y}_p(\bm{x}) = \bm{0}$ for every $\bm{x} \in \mathcal{N}_2$ which, in turn, yields
\begin{align}
\| \bm{u}_p (\bm{x}) \| 
&\leq \sigma \| \bm{F}_{p,[m]} (\bm{0}, \bm{0}, \bm{x}) - \bm{F}_{p,[m]} (\bm{0}, \bm{0}, \bm{x}^*) \|.
\end{align}
Observing that
\begin{align}
\bm{F}_{p,[m]} (\bm{0}, \bm{0}, \bm{x}) &= \begin{bmatrix}
\nabla \bm{h} (\bm{x}) \bm{M}(\bm{x}) \bm{0} + \frac{1}{\sqrt{p}} \bm{0} + \bm{h} (\bm{x}) \\
\bm{g} \left( \bm{x} + \bm{M}(\bm{x}) \bm{0} \right) \\
\bm{0}
\end{bmatrix} = \begin{bmatrix}
\bm{h} \left( \bm{x} \right) \\
\bm{g} \left( \bm{x} \right) \\
\bm{0}
\end{bmatrix}
\end{align}
and
\begin{align}
\bm{F}_{p,[m]} (\bm{0}, \bm{0}, \bm{x}^*) &= \begin{bmatrix}
\bm{h} \left( \bm{x}^* \right) \\
\bm{g} \left( \bm{x}^* \right) \\
\bm{0}
\end{bmatrix} = \begin{bmatrix}
\bm{0} \\
\bm{0} \\
\bm{0}
\end{bmatrix}
\end{align}
we now have that
\begin{align}
\| \bm{u}_p (\bm{x}) \| &\leq \sigma \left\| \begin{bmatrix}
\bm{h} \left( \bm{x} \right) \\
\bm{g} \left( \bm{x} \right)
\end{bmatrix} \right\|, \label{NewtonYThm.4}
\end{align}
for all $\bm{x} \in \mathcal{N}_2$.

Now let $\eta >0$ be given such that $\bm{h}, \bm{g} \in \mathcal{C}^2 (\mathcal{B} (\bm{x}^*, \eta))$ and $\mathcal{B} (\bm{x}^*, \eta) \subset \mathcal{N}_2$. Define $\mathcal{B} := \mathcal{B} (\bm{x}^*, \eta)$. Let
\begin{align}
\eta &= \max \left\{ \max_{\bm{x} \in \mathcal{B}} \| \bm{M}(\bm{x}) \|, 1 \right\}, \\
\delta &= \left[ \sum_{j = 1}^{\ell} \max_{\bm{x}_j \in \mathcal{B}} \| \nabla^2 \bm{h}_j ( \bm{x}_j ) \|^2 \right]^{1/2},
\end{align} 
and let $\bm{w}_0 \in \mathcal{B} \cap \Omega$ be any point such that both
\begin{align}
\eta \sigma \gamma \| \bm{h}(\bm{w}_0) \| \leq \min \left\{ 1/2, 1 - \tau, \gamma (1 - \alpha) \right\} \label{NewtonYThm.4.1}
\end{align}
and $\mathcal{B} ( \bm{w}_0, 2 \eta \sigma \gamma \| \bm{h}(\bm{w}_0) \|) \subseteq \mathcal{B} \cap \Omega$, where $\displaystyle \gamma = \left( 1 + \frac{\delta \eta \sigma}{2} \right)$.

Suppose that $\bm{w}_i \in \mathcal{B} \cap \Omega$ and that $\bm{w}_{i+1}$ is generated using the scheme in (\ref{NewtonYScheme}) with $p_i$ chosen such that $p_i \geq \max \left\{ \beta^2, \| \bm{h} (\bm{w}_i) \|^{-2} \right\}$, for some $\beta > 1$, and with $s_i = 1$. By our definition of $\bm{F}_{p,[m]} (\bm{u}, \bm{y}, \bm{x})$ and the fact that $\bm{F}_{p_i,[m]} (\bm{u}_{p_i} (\bm{w}_i), \bm{0}, \bm{w}_i) = \bm{0}$, we have that 
\begin{align}
\nabla \bm{h} (\bm{w}_i) \bm{M}(\bm{w}_i) \bm{u}_{p_i} (\bm{w}_i) = - \bm{h} (\bm{w}_i)
\end{align}
and
\begin{align}
\bm{g} \left(\bm{w}_i + \bm{M}(\bm{w}_i) \bm{u}_{p_i} (\bm{w}_i) \right) = \bm{0}
\ \ \ \ \Longleftrightarrow \ \ \ \ 
\bm{w}_i + \bm{M}(\bm{w}_i) \bm{u}_{p_i} (\bm{w}_i) \in \Omega.
\label{NewtonYThm.4.1.0}
\end{align}
Hence, the vector 
$\begin{bmatrix} \bm{w}_i + \bm{M}(\bm{w}_i) \bm{u}_{p_i} (\bm{w}_i), \ \bm{0} \end{bmatrix}^{\intercal}$ satisfies the constraints of the minimization problem in (\ref{NewtonYScheme}). Thus, the objective value of problem (\ref{NewtonYScheme}) at a solution, $\bm{w}_{i+1}$, is bounded above by the objective evaluated at $\begin{bmatrix} \bm{w}_i + \bm{M}(\bm{w}_i) \bm{u}_{p_i} (\bm{w}_i), \ \bm{0} \end{bmatrix}^{\intercal}$. Using this fact with the definition of $\eta$ and inequality (\ref{NewtonYThm.4}) we have that
\begin{align}
\| \bm{w}_{i+1} - \bm{w}_i \|^2 + \| \bm{y}_{i+1} \|^2 
&\leq \| \bm{M}(\bm{w}_i) \bm{u}_{p_i} (\bm{w}_i) \|^2
\leq \eta^2 \| \bm{u}_{p_i} (\bm{w}_i) \|^2
%\leq \eta^2 \sigma^2 \left\| \begin{bmatrix}
%\bm{h} \left( \bm{w}_i \right) \\
%\bm{g} \left( \bm{w}_i \right)
%\end{bmatrix} \right\|
\leq \eta^2 \sigma^2 \left\| \bm{h}(\bm{w}_i) \right\|^2. \label{NewtonYThm.4.1.1}
\end{align}
In particular, recalling that $\alpha_i := 1 - \| \bm{y}_i \|$, it follows from (\ref{NewtonYThm.4.1.1}) that
\begin{align}
1 - \eta \sigma \| \bm{h} (\bm{w}_i) \| \leq \alpha_{i+1}. \label{NewtonYThm.4.1.2}
\end{align}

Now suppose that $\bm{w}_i, \bm{w}_{i+1} \in \mathcal{B}$ where $\bm{w}_{i+1}$ is generated using the scheme in (\ref{NewtonYScheme}) with $p_i$ chosen such that $p_i \geq \max \left\{ \beta^2, \| \bm{h} (\bm{w}_i) \|^{-2} \right\}$, for some $\beta > 1$, and with $s_i = 1$. In particular, since $\bm{w}_i$ is known at the time of choosing $p_i$ we can take
\begin{align}
p_i = \max \left\{ \beta^2, \frac{1}{\| \bm{h}(\bm{w}_i) \|^2} \right\} \label{NewtonYThm.4.2}
\end{align}
for some $\beta \geq 1$ to satisfy this requirement. By Taylor's Theorem and the problem constraint, we have that 
\begin{align}
\bm{h}(\bm{w}_{i+1}) &= \bm{h} (\bm{w}_i) + \nabla \bm{h} (\bm{w}_i) (\bm{w}_{i+1} - \bm{w}_i) + \frac{1}{2} \sum_{j = 1}^{\ell} ( \bm{w}_{i+1} - \bm{w}_i )^{\intercal} \nabla^2 \bm{h}_j (\bm{\xi}_j) ( \bm{w}_{i+1} - \bm{w}_i ) \bm{e}_j, \\
&= - \frac{1}{\sqrt{p_i}} \bm{y}_{i+1} + \frac{1}{2} \sum_{j = 1}^{\ell} ( \bm{w}_{i+1} - \bm{w}_i )^{\intercal} \nabla^2 \bm{h}_j (\bm{\xi}_j) ( \bm{w}_{i+1} - \bm{w}_i ) \bm{e}_j,
\end{align}
for some vectors $\bm{\xi}_j$ between $\bm{w}_i$ and $\bm{w}_{i+1}$, for $1 \leq j \leq \ell$. This yields the inequality
\begin{align}
\| \bm{h}(\bm{w}_{i+1}) \| \leq \frac{1}{\sqrt{p_i}} \| \bm{y}_{i+1} \| + \frac{\delta}{2} \| \bm{w}_{i+1} - \bm{w}_i \|^2. \label{NewtonYThm.5}
\end{align}
Combining (\ref{NewtonYThm.5}) with inequality (\ref{NewtonYThm.4.1.1}) yields
\begin{align}
\| \bm{h}(\bm{w}_{i+1}) \| \leq \frac{ \eta \sigma}{\sqrt{p_i}} \left\| \bm{h}(\bm{w}_i) \right\| + \frac{ \delta \eta^2 \sigma^2}{2} \left\| \bm{h}(\bm{w}_i) \right\|^2. \label{NewtonYThm.6}
\end{align}
By our choice of $p_i$ in (\ref{NewtonYThm.4.2}), we have that $\frac{1}{\sqrt{p_i}} \leq \| \bm{h} (\bm{w}_i) \|$. Combining this with (\ref{NewtonYThm.6}) yields
\begin{align}
\| \bm{h}(\bm{w}_{i+1}) \| \leq \eta \sigma \left( 1 + \frac{\delta \eta \sigma}{2} \right) \left\| \bm{h}(\bm{w}_i) \right\|^2 = \eta \sigma \gamma \left\| \bm{h}(\bm{w}_i) \right\|^2, \label{NewtonYThm.7}
\end{align}
which concludes the proof of (\ref{NewtonYSchemeThm.3}).

It remains to show that the sequence of iterates $\{ \bm{w}_i \}_{i = 0}^{\infty}$ generated by (\ref{NewtonYScheme}) is contained in the ball $\mathcal{B}$ and that $\alpha_i \geq \alpha$, for all $i \geq 1$. We proceed by way of induction. First, $\bm{w}_0 \in \mathcal{B}$ by definition. Additionally, using (\ref{NewtonYThm.4.1.2}) and (\ref{NewtonYThm.4.1}) we have that
\begin{align}
\alpha_1
\geq 1 - \eta \sigma \| \bm{h} (\bm{w}_0) \|
\geq \alpha.
\end{align} 
Hence, the base case of induction holds. Suppose now that $\bm{w}_0, \bm{w}_1, \ldots, \bm{w}_i$ are contained in $\mathcal{B}$, that $\| \bm{w}_{i} - \bm{w}_{i-1} \| \leq \eta \sigma \| \bm{h}( \bm{w}_0 ) \| 2^{-2^{i-1} + 1}$, and that $s_j = 1$ and $\alpha_{j+1} \geq 0$, for $0 \leq j < i$. As the inequality in (\ref{NewtonYThm.7}) holds for $\bm{w}_0, \bm{w}_1, \ldots, \bm{w}_i$ by the inductive hypothesis, repeated applications of (\ref{NewtonYThm.7}) yield
\begin{align}
\| \bm{w}_{i+1} - \bm{w}_i \| &\leq \eta \sigma \left\| \bm{h}(\bm{w}_i) \right\| \leq \frac{1}{\gamma} \left( \eta \sigma \gamma \| \bm{h}( \bm{w}_0 ) \| \right)^{2^i} = \eta \sigma \| \bm{h}( \bm{w}_0 ) \| \left( \eta \sigma \gamma \| \bm{h}( \bm{w}_0 ) \| \right)^{2^i - 1}. \label{NewtonYThm.8}
\end{align}
By our choice of $\bm{w}_0$ in (\ref{NewtonYThm.4.1}), it follows from (\ref{NewtonYThm.8}) that
\begin{align}
\| \bm{w}_{i+1} - \bm{w}_i \| &\leq \eta \sigma \left\| \bm{h}(\bm{w}_0) \right\| 2^{-2^i + 1}, \label{NewtonYThm.8.1}
\end{align}
as desired. Next, observe that (\ref{NewtonYSchemeThm.1}) follows from applying (\ref{NewtonYThm.8.1}), the inductive hypothesis, and the triangle inequality
\begin{align}
\| \bm{w}_{i+1} - \bm{w}_0 \| = \| \bm{w}_{i+1} - \bm{w}_i + \bm{w}_i - \bm{w}_0 \| &\leq \| \bm{w}_{i+1} - \bm{w}_i \| + \| \bm{w}_{i} - \bm{w}_0 \| \\
&\leq \sum_{j = 0}^{i} \| \bm{w}_{j+1} - \bm{w}_j \| \\
&\leq \eta \sigma \left\| \bm{h}(\bm{w}_0) \right\| \sum_{j = 0}^{i} 2^{-2^j + 1} \\
&< 2 \eta \sigma \left\| \bm{h}(\bm{w}_0) \right\|.
\end{align}
Additionally, observe that
\begin{align}
\| \bm{h}(\bm{w}_{i+1}) \| &\leq \eta \sigma \| \bm{h}(\bm{w}_i) \|^2 \leq \left( \eta \sigma \| \bm{h}(\bm{w}_0) \| \right)^{2^i} \| \bm{h}(\bm{w}_i) \| \leq \left( 1 - \tau \right) \| \bm{h}(\bm{w}_i) \|, \label{NewtonYThm.9}
\end{align}
where the final inequality holds from our choice of $\bm{w}_0$ in (\ref{NewtonYThm.4.1}). By our choice of $\bm{w}_0$ in (\ref{NewtonYThm.4.1}), we have that
\begin{align}
\eta \sigma \| \bm{h} (\bm{w}_0) \| 2^{-2^i + 1} \leq 1 - \alpha, \label{NewtonYThm.9.1}
\end{align}
for all $i \geq 0$. Hence, using (\ref{NewtonYThm.4.1.2}), (\ref{NewtonYThm.8}), and (\ref{NewtonYThm.9.1}) yields
\begin{align}
\alpha_{i+1} 
&\geq 1 - \eta \sigma \| \bm{h} (\bm{w}_i) \| %\tag{Using (\ref{NewtonYThm.4.1.2})} \\
\geq 1 - \eta \sigma \| \bm{h} (\bm{w}_0) \| 2^{-2^i + 1} %\tag{Using (\ref{NewtonYThm.8})} \\
\geq \alpha, %\tag{Using (\ref{NewtonYThm.9.1})}
\end{align}
as desired. As $\alpha_{i+1} \leq 1$ by definition, we have $\alpha_{i+1} \in [\alpha, 1]$. Hence, $1 - \tau \leq 1 - \tau \alpha_i \leq 1 - \tau \alpha$. Combining this with (\ref{NewtonYThm.9}) yields that the Armijo line search update in (\ref{NewtonArmijo}) is satisfied for $s_i = 1$. This concludes the inductive phase of the proof.

Lastly, as $1 - \tau \leq 1 - \tau \alpha$ we have that
\begin{align}
\| \bm{h}(\bm{w}_{i+1}) \| 
&\leq \left( 1 - \tau \alpha \right) \| \bm{h}(\bm{w}_i) \|, \label{NewtonYThm.10}
\end{align}
holds for all $i \geq 0$. Hence the sequence $\{ \bm{w}_i \}_{i = 0}^{\infty}$ is Cauchy so there exists a vector $\bm{w}^* \in \mathcal{B}$ such that $\bm{w}_i \to \bm{w}^*$ as $i \to \infty$ which, combined with (\ref{NewtonYThm.9}), yields the inequality in (\ref{NewtonYSchemeThm.2}). Lastly, as $\tau \in (0,1)$, it follows from (\ref{NewtonYThm.10}) that $\| \bm{h} (\bm{w}^*) \| = 0$. As $\bm{g}(\bm{w}_i) = \bm{0}$ for all $i \geq 0$, we conclude that (\ref{NewtonYSchemeThm.0}) holds.
\end{proof}

\subsection{Convergence Analysis for Multiplier Step} \label{subsec:multiplier-step-analysis} 
We now work towards establishing local convergence for the multiplier step in the LS algorithm. Before stating and proving the local convergence result for the multiplier step, we require seceral results that will be used in the proof. First, we restate Theorem~3.1 and Corollary~3.1 from \cite{Hager1993} in terms of our current analysis. Note that there are no significant changes to the details of these results and the same proofs from \cite{Hager1993} holds. As such, the proof is omitted here. %% NOTE: This result actually uses LICQ instead of LI like the old result. However, based on our analysis for Theorem 4.1 the idea used there can be carried over to the original proof of THeorem 4.2. 

\begin{theorem} \label{Theorem3.1}
Suppose $\bm{x}^*$ satisfies $\bm{h}(\bm{x}^*) = \bm{0}$, $\bm{x}^* \in \Omega$, $\bm{h} \in \mathcal{C}^2$ near $\bm{x}^*$, and assumptions (\textbf{LICQ}) and (\textbf{SOSC}) hold.
Then there exists a neighborhood $\mathcal{N}$ about $(\bm{x}^*, \bm{\lambda}^*)$ for which the problem 
\begin{align}
\min \left\{ \mathcal{L}_{p_i}(\bm{z}, \bm{\nu}_i) : \nabla \bm{h} (\bm{u}_i) (\bm{z} - \bm{u}_i) = \bm{0}, \bm{z} \in \Omega \right\}
\end{align}
has a local minimizer $\bm{z} = \bm{z}_{i+1}$, whenever $(\bm{u}_i, \bm{\nu}_i) \in \mathcal{N}$. Moreover, there exists a constant $c \in \mathbb{R}$ such that
\begin{align}
\| \bm{z}_{i+1} - \bm{x}^* \| \leq c \| \bm{\nu}_i - \bm{\lambda}^* \|^2 + c \| \bm{u}_i - \bm{x}^* \|^2 + c \| \bm{h}(\bm{u}_i) \|,
\end{align} 
for every $(\bm{u}_i, \bm{\nu}_i) \in \mathcal{N}$ with $\bm{u}_i \in \Omega$.
\end{theorem}

%Next, we restate Corollary 3.1 from \cite{Hager1993} in terms of our current analysis (note that there are no changes to the details of the result and the same proof from \cite{Hager1993} holds).

\begin{corollary} \label{Corollary3.1}
Suppose $\bm{x}^*$ satisfies $\bm{h} (\bm{x}^*) = \bm{0}$, $\bm{x}^* \in \Omega$, $\bm{h} \in \mathcal{C}^2$ near $\bm{x}^*$, and assumptions (\textbf{LICQ}) and (\textbf{SOSC}) hold.
Then there exists a neighborhood $\mathcal{N}$ about $(\bm{x}^*, \bm{\lambda}^*)$ for which the problem 
\begin{align}
\min \left\{ \mathcal{L}_{p_i}(\bm{z}, \bm{\nu}_i) : \nabla \bm{h} (\bm{u}_i) (\bm{z} - \bm{u}_i) = \bm{0}, \bm{z} \in \Omega \right\}
\end{align}
has a local minimizer $\bm{z} = \bm{z}_{i+1}$, whenever $(\bm{x}_k, \bm{\nu}_i) \in \mathcal{N}$, $\bm{x}_k \in \Omega$, and any $\bm{u}_i = \bm{w}_j$ for any $j \geq 1$, where $\bm{w}_j$ are generated by the scheme in (\ref{NewtonYScheme}) starting from $\bm{w}_0 = \bm{x}_k$. Moreover, there exists a constant $c \in \mathbb{R}$ such that
\begin{align}
\| \bm{z}_{i+1} - \bm{x}^* \| \leq c \| \bm{\nu}_i - \bm{\lambda}^* \|^2 + c \| \bm{x}_k - \bm{x}^* \|^2,
\end{align} 
where $c$ is independent of $(\bm{x}_k, \bm{\nu}_i) \in \mathcal{N} \cap \left( \Omega \times \mathbb{R}^m \right)$.
\end{corollary}
\iffalse
\begin{proof}
As $\bm{h} (\bm{x}^*) = \bm{0}$ we have that
\begin{align}
\| \bm{h}(\bm{x}_k)  \| = \| \bm{h}(\bm{x}_k) - \bm{h}(\bm{x}^*) \| = \| \nabla \bm{h}(\bm{\xi}) (\bm{x}_k - \bm{x}^*) \| \leq c \| \bm{x}_k - \bm{x}^* \|, \label{Corollary.1}
\end{align}
where $\bm{\xi}$ is some vector between $\bm{x}_k$ and $\bm{x}^*$. Using (\ref{NewtonYSchemeThm.1}) from Theorem~\ref{NewtonYSchemeThm} with $\bm{w}_0 = \bm{x}_k$, we have that $\| \bm{w}_j - \bm{x}_k \| \leq c \| \bm{h}(\bm{x}_k) \|$, for all $j \geq 1$. As $\bm{u}_i = \bm{w}_j$ for some $1 \leq j \leq i_C$, it follows that $\| \bm{u}_i - \bm{x}_k \| \leq c \| \bm{h}(\bm{x}_k) \|$. Combining this with (\ref{Corollary.1}) and the triangle inequality yields
\begin{align*}
\| \bm{u}_i - \bm{x}^* \| 
&\leq \| \bm{u}_i - \bm{x}_k \| + \| \bm{x}_k - \bm{x}^* \| 
\leq c \| \bm{x}_k - \bm{x}^* \|.
\end{align*}
Lastly, (\ref{NewtonYSchemeThm.3}) from Theorem~\ref{NewtonYSchemeThm} with $\bm{w}_0 = \bm{x}_k$ yields that $\| \bm{h}(\bm{u}_i) \| \leq c \| \bm{h}(\bm{x}_k) \|^2$ which combined with (\ref{Corollary.1}) yields that $\| \bm{h}(\bm{u}_i) \| \leq c \| \bm{x}_k - \bm{x}^* \|^2$. Making these substitutions into Theorem~\ref{Theorem3.1} concludes the proof.
\end{proof}
\fi

In order to prove the local convergence result for the multiplier step, we make use of the following problem:
\begin{align}
\begin{array}{cc}
\displaystyle \min_{\bm{x} \in \mathbb{R}^{n}} & f(\bm{x}) \\
\text{s.t.} & \bm{h}(\bm{x}) = \bm{u}, \ \ \bm{r} (\bm{x}) \leq \bm{0},
\end{array} \label{prob:Ru}
\end{align}
%%%%%%%%%%%%%%%%%%%%%%%%%%%%%%%%%%%
\iffalse
\begin{align}
\begin{array}{cc}
\displaystyle \min_{\bm{x}} & f(\bm{x}) \\
\text{s.t.} & \bm{h} (\bm{x}) = \bm{u}, \ \ \bm{g} (\bm{x}) = \bm{0},
\end{array} \label{prob:Gu}
\end{align}
and
\fi
%%%%%%%%%%%%%%%%%%%%%%%%%%%%%%%%%%%
where $\bm{u}$ is a fixed vector in $\mathbb{R}^m$. Note that (\ref{prob:Ru}) is very similar to problem (\ref{prob:main-nlp}) with the only difference being the perturbation applied to the nonlinear portion of the constraint set, namely changing $\bm{h} (\bm{x}) = \bm{0}$ to $\bm{h} (\bm{x}) = \bm{u}$. Lemma~\ref{lem:Lem13.1} will provide insight into how close a solution of problem (\ref{prob:Ru}) is to a solution of (\ref{prob:main-nlp}) for different values of $\bm{u}$ by performing stability analysis of problem (\ref{prob:Ru}) for values of $\bm{u}$ in a neighborhood of $\bm{u} = \bm{0}$. 
%TODO
%We note that Lemma~\ref{lem:Lem13.1} has the same conclusions as Lemma 13.1 from \cite{Hager1993}, however, the difference here being that the hypotheses are more general. This is accomplished by making use of strong regularity as discussed by Robinson in \cite{Robinson1980}. This difference will allow us to clearly justify alternate sets of assumptions when establishing convergence results for the multiplier step of NPASA. 

\begin{lemma} \label{lem:Lem13.1}
Suppose that $f, \bm{h}$ from problem (\ref{prob:main-nlp}) are twice continuously differentiable and that %the KKT system for problem (\ref{prob:Ru}) is strongly regular at a KKT point $(\bm{x}^*, \bm{\lambda}^*, \bm{\mu}^*)$. 
either assumptions (\textbf{LICQ}), (\textbf{SCS}), and (\textbf{SOSC}) are satisfied at $(\bm{x}^*, \bm{\lambda}^*, \bm{\mu}^*)$ or assumptions (\textbf{LICQ}) and (\textbf{SSOSC}) are satisfied at $(\bm{x}^*, \bm{\lambda}^*, \bm{\mu}^*)$. Then there exists a neighborhood $\mathcal{N}$ of $\bm{0}$ such that, for every $\bm{u} \in \mathcal{N}$, problem (\ref{prob:Ru}) has a local minimizer $\bm{x}(\bm{u})$ and associated multipliers $\bm{\lambda} (\bm{u})$ and $\bm{\mu} (\bm{u})$ satisfying
\iffalse
\begin{itemize}
\item[(\ref{lem:Lem13.1}.1)] $\begin{pmatrix} \bm{x}(\bm{0}), \ \bm{\lambda} (\bm{0}), \ \bm{\mu} (\bm{0}) \end{pmatrix} = \left( \bm{x}^*, \bm{\lambda}^*, \bm{\mu}^* \right)$.
\item[(\ref{lem:Lem13.1}.2)] $\begin{pmatrix} \bm{x}(\bm{u}), \ \bm{\lambda} (\bm{u}), \ \bm{\mu} (\bm{u}) \end{pmatrix}$ satisfies the KKT system for problem (\ref{prob:Ru}).
\end{itemize}
\fi
\begin{align}
&\begin{pmatrix} \bm{x}(\bm{0}), \ \bm{\lambda} (\bm{0}), \ \bm{\mu} (\bm{0}) \end{pmatrix} = \left( \bm{x}^*, \bm{\lambda}^*, \bm{\mu}^* \right),\label{lem:Lem13.1.1} \\
&\begin{pmatrix} \bm{x}(\bm{u}), \ \bm{\lambda} (\bm{u}), \ \bm{\mu} (\bm{u}) \end{pmatrix} \ \text{satisfies the KKT system for problem (\ref{prob:Ru})}. \label{lem:Lem13.1.2}
\end{align}
Additionally, there exists a constant $c \in \mathbb{R}$ such that, for every $\bm{u}_1, \bm{u}_2 \in \mathcal{N}$, we have
\iffalse
\begin{itemize}
\item[(\ref{lem:Lem13.1}.3)] $\displaystyle \left\| \left(
\bm{x} (\bm{u}_1), \bm{\lambda} (\bm{u}_1), \bm{\mu} (\bm{u}_1) 
\right) - \left(
\bm{x} (\bm{u}_2), \bm{\lambda} (\bm{u}_2), \bm{\mu} (\bm{u}_2)
\right) \right\| \leq c \| \bm{u}_1 - \bm{u}_2 \|$.
\end{itemize}
\fi
\begin{align}
\left\| \left(
\bm{x} (\bm{u}_1), \bm{\lambda} (\bm{u}_1), \bm{\mu} (\bm{u}_1) 
\right) - \left(
\bm{x} (\bm{u}_2), \bm{\lambda} (\bm{u}_2), \bm{\mu} (\bm{u}_2)
\right) \right\| \leq c \| \bm{u}_1 - \bm{u}_2 \|. \label{lem:Lem13.1.3}
\end{align}
\end{lemma}
\begin{proof}
We first establish that the KKT system for problem (\ref{prob:main-nlp}) is strongly regular at $(\bm{x}^*, \bm{\lambda}^*, \bm{\mu}^*)$. In Section 4 of \cite{Robinson1980}, Robinson outlines conditions under which the KKT system for a nonlinear program is strongly regular. One set of assumptions that guarantees a strongly regular KKT system for problem (\ref{prob:main-nlp}) at $(\bm{x}^*, \bm{\lambda}^*, \bm{\mu}^*)$ is that (\textbf{LICQ}), (\textbf{SCS}), and (\textbf{SOSC}) hold at $(\bm{x}^*, \bm{\lambda}^*, \bm{\mu}^*)$. Additionally, Theorem 4.1 of \cite{Robinson1980} proves that if assumptions (\textbf{LICQ}) and (\textbf{SSOSC}) hold at $(\bm{x}^*, \bm{\lambda}^*, \bm{\mu}^*)$ then this guarantees strong regularity for the KKT system in (\ref{prob:main-nlp}). Under either set of hypotheses for our lemma, we have that the KKT system for problem (\ref{prob:main-nlp}) is strongly regular at $(\bm{x}^*, \bm{\lambda}^*, \bm{\mu}^*)$. We now prove the lemma using some results from \cite{Robinson1980}.

Now denote the Lagrangian of problem (\ref{prob:Ru}) by
\begin{align}
\mathcal{L} (\bm{u}; \bm{x}, \bm{\lambda}, \bm{\mu}) = f (\bm{x}) + \bm{\lambda}^{\intercal} \left( \bm{h} (\bm{x}) - \bm{u} \right) + \bm{\mu}^{\intercal} \bm{r} (\bm{x}). \label{eq:KKTStability.0}
\end{align}
Using $\bm{f}(\cdot)$ in this proof in the context of \cite{Robinson1980}, we define 
\begin{align}
\bm{f} (\bm{u}, \bm{x}, \bm{\lambda}, \bm{\mu}) 
:= \begin{bmatrix}
\nabla_x \mathcal{L} (\bm{x}, \bm{\lambda}, \bm{\mu}) \\
%- \bm{r} (\bm{x}) \\
- \bm{r} (\bm{x}) \\
- \bm{h} (\bm{x}) + \bm{u}
\end{bmatrix}. \label{eq:KKTStability.1}
\end{align}
By the hypotheses of our lemma, we have that the hypotheses of Theorem 2.1 of \cite{Robinson1980} are satisfied. To apply Corollary 2.2 from \cite{Robinson1980}, it remains to show that there exists a constant $\nu$ such that
\begin{align}
\| \bm{f} (\bm{u}_1, \bm{x}, \bm{\lambda}, \bm{\mu}) - \bm{f} (\bm{u}_2, \bm{x}, \bm{\lambda}, \bm{\mu}) \| \leq \nu \| \bm{u}_1 - \bm{u}_2 \|, \label{eq:KKTStability.2}
\end{align}
for all $\bm{u}_1, \bm{u}_2 \in \mathcal{N}$ and $(\bm{x}, \bm{\lambda}, \bm{\mu}) \in V$, where $V$ is a closed neighborhood of $(\bm{x}^*, \bm{\lambda}^*, \bm{\mu}^*)$. Noting that
\begin{align}
\nabla_x \mathcal{L} (\bm{u}; \bm{x}, \bm{\lambda}, \bm{\mu})^{\intercal} = \nabla_x f (\bm{x})^{\intercal} + \nabla_x \bm{h} (\bm{x})^{\intercal} \bm{\lambda} + \nabla_x \bm{r} (\bm{x})^{\intercal} \bm{\mu} \label{eq:KKTStability.3}
\end{align}
it follows that
\begin{align}
\| \bm{f} (\bm{u}_1, \bm{x}, \bm{\lambda}, \bm{\mu}) - \bm{f} (\bm{u}_2, \bm{x}, \bm{\lambda}, \bm{\mu}) \| 
= \left\| \begin{bmatrix}
\bm{0} \\
\bm{0} \\
\bm{u}_1 - \bm{u}_2 
\end{bmatrix} \right\|
= \left\| \bm{u}_1 - \bm{u}_2 \right\|, \label{eq:KKTStability.4}
\end{align}
which concludes the proof.
\end{proof}

%In Section 4 of \cite{Robinson1980}, Robinson outlines sufficient conditions under which the KKT system for a nonlinear program is strongly regular. One set of assumptions that guarantees a strongly regular KKT system is (\textbf{LICQ}), (\textbf{SCS}), and (\textbf{SOSC}). Additionally, Theorem 4.1 of \cite{Robinson1980} proves that assumptions (\textbf{LICQ}) and (\textbf{SSOSC}) guarantee strong regularity for the KKT system in a nonlinear program.

Next, we consider the problem where we %choose $j \in \{0,1\}$ and 
fix a vector $\bm{u}$ then solve
\begin{align}
\begin{array}{cc}
\displaystyle \min_{\bm{\lambda}, \bm{\mu}} & E_{m,0} (\bm{u}, \bm{\lambda}, \bm{\mu}) + \gamma \left\| [\bm{\lambda}, \bm{\mu}] \right\|^2 \\
\text{s.t.} & \bm{\mu} \geq \bm{0}.
\end{array} \label{prob:Mu-gamma}
\end{align}
where $\gamma > 0$ is a constant. This problem appears in the LS algorithm and it will simplify the upcoming convergence analysis for the multiplier step of the LS algorithm to consider it individually beforehand. For this problem, we have a stability result that follows from an application of Corollary 2.2 in \cite{Robinson1980}.

\begin{lemma} \label{lem:AltLem13.2}
Suppose that $f, \bm{h} \in \mathcal{C}^1$ and that $(\bm{x}^*, \bm{\lambda}^*, \bm{\mu}^*)$ is a KKT point for problem (\ref{prob:main-nlp}). Then there exists a neighborhood $\mathcal{N}$ of $\bm{x}^*$ such that, for every $\bm{u} \in \mathcal{N}$, problem (\ref{prob:Mu-gamma}) has a local minimizer $(\bm{\lambda} (\bm{u}), \bm{\mu} (\bm{u}))$ and associated multiplier $\bm{\alpha} (\bm{u})$ satisfying
\iffalse
\begin{itemize}
\item[(\ref{lem:AltLem13.2}.1)] $\begin{pmatrix} \bm{\lambda} (\bm{x}^*), \ \bm{\mu} (\bm{x}^*) \end{pmatrix} = \left( \bm{\lambda}^*, \bm{\mu}^* \right)$.
\item[(\ref{lem:AltLem13.2}.2)] $\begin{pmatrix} \bm{\lambda} (\bm{u}), \ \bm{\mu} (\bm{u}), \ \bm{\alpha} (\bm{u}) \end{pmatrix}$ satisfies the KKT system for problem (\ref{prob:Mu-gamma}).
\end{itemize}
\fi
\begin{align}
&\begin{pmatrix} \bm{\lambda} (\bm{x}^*), \ \bm{\mu} (\bm{x}^*) \end{pmatrix} = \left( \bm{\lambda}^*, \bm{\mu}^* \right), \label{lem:AltLem13.2.1} \\
&\begin{pmatrix} \bm{\lambda} (\bm{u}), \ \bm{\mu} (\bm{u}), \ \bm{\alpha} (\bm{u}) \end{pmatrix} \ \text{satisfies the KKT system for problem (\ref{prob:Mu-gamma})}. \label{lem:AltLem13.2.2}
\end{align}
Additionally, there exists a constant $c \in \mathbb{R}$ such that, for every $\bm{u}_1, \bm{u}_2 \in \mathcal{N}$, we have
\iffalse
\begin{itemize}
\item[(\ref{lem:AltLem13.2}.3)] $\displaystyle \left\| \left(
\bm{\lambda} (\bm{u}_1), \bm{\mu} (\bm{u}_1), \bm{\alpha} (\bm{u}_1)
\right) - \left(
\bm{\lambda} (\bm{u}_2), \bm{\mu} (\bm{u}_2), \bm{\alpha} (\bm{u}_2)
\right) \right\| \leq c \| \bm{u}_1 - \bm{u}_2 \|$.
\end{itemize}
\fi
\begin{align}
\left\| \left(
\bm{\lambda} (\bm{u}_1), \bm{\mu} (\bm{u}_1), \bm{\alpha} (\bm{u}_1)
\right) - \left(
\bm{\lambda} (\bm{u}_2), \bm{\mu} (\bm{u}_2), \bm{\alpha} (\bm{u}_2)
\right) \right\| \leq c \| \bm{u}_1 - \bm{u}_2 \| \label{lem:AltLem13.2.3}
\end{align}
\end{lemma}

\begin{proof}
Denote the Lagrangian for problem (\ref{prob:Mu-gamma}) by
\begin{align}
\mathcal{L} (\bm{u}; \bm{\lambda}, \bm{\mu}, \bm{\alpha}) = E_{m,0} (\bm{u}, \bm{\lambda}, \bm{\mu}) + \gamma \left\| [\bm{\lambda}, \bm{\mu}] \right\|^2 - \bm{\mu}^{\intercal} \bm{\alpha}. \label{eq:EMStability.0.0}
\end{align}
As in the proof of Lemma~\ref{lem:Lem13.1}, we first establish that the KKT system for problem (\ref{prob:Mu-gamma}) is strongly regular at $(\bm{\lambda}^*, \bm{\mu}^*, \bm{\alpha}^*)$ under the given hypotheses. Using (\ref{eq:EMStability.0.0}) we have that the hessian of the Lagrangian for problem (\ref{prob:Mu-gamma}) is given by $\nabla^2 \mathcal{L} (\bm{u}; \bm{\lambda}, \bm{\mu}, \bm{\alpha}) = B (\bm{u})^{\intercal} B (\bm{u}) + \gamma \bm{I}$, where $\bm{I}$ is the identity matrix and $B (\bm{u})$ is the matrix defined by
\begin{align}
B (\bm{u})
&:= \begin{bmatrix}
\ \nabla_x \bm{h} (\bm{u})^{\intercal} &\vert& \nabla_x \bm{r} (\bm{u})^{\intercal} \
\end{bmatrix}. \label{eq:EmDiscussion.1}
\end{align}
Hence, for every vector $\bm{v} = [\bm{\lambda}, \bm{\mu}]^{\intercal}$ with $\bm{v} \neq \bm{0}$ we have
\begin{align}
\bm{v}^{\intercal} \nabla^2 \mathcal{L} (\bm{u}; \bm{\lambda}, \bm{\mu}, \bm{\alpha}) \bm{v}
&= \left\| B (\bm{u}) \bm{v} \right\|^2 + \gamma \left\| \bm{v} \right\|^2
> 0. \label{eq:EmDiscussion.2}
\end{align}
So we have that the hessian of the Lagrangian for problem (\ref{prob:Mu-gamma}) is positive definite. 
\iffalse %%%%%%%%%%%%%%%%%%%%%%%%%%%%%%
Under our hypothesis, %that (\textbf{LI}) holds at $\bm{x}^*$, 
we have that $B (\bm{x}^*)$ is of full column rank yielding that $B (\bm{x}^*)^{\intercal} B (\bm{x}^*)$ is nonsingular. As such, there exists a constant $\sigma > 0$ such that
\begin{align}
\bm{v}^{\intercal} \nabla^2 \mathcal{L} (\bm{x}^*; \bm{\lambda}, \bm{\mu}, \bm{\alpha}) \bm{v}
&= \left\| B (\bm{x}^*) \bm{v} \right\|^2
> \sigma. \label{eq:EmDiscussion.3}
\end{align}
So we have positive definiteness of the hessian. %under assumption (\textbf{LI}). 
\fi %%%%%%%%%%%%%%%%%%%%%%%%%%%%%%
As (\ref{prob:Mu}) only has bound constraints, the gradients of the active constraints are linearly independent. Thus, %under the assumption that (\textbf{LI}) holds at $\bm{x}^*$ 
the KKT system in problem (\ref{prob:Mu-gamma}) satisfies the strong regularity condition from \cite{Robinson1980}.

Now define $\bm{f} (\bm{u}, \bm{\lambda}, \bm{\mu}, \bm{\alpha}) := \begin{bmatrix}
\nabla_{[\lambda, \mu]} \mathcal{L} (\bm{u}; \bm{\lambda}, \bm{\mu}, \bm{\alpha})^{\intercal} \\
\bm{\mu}
\end{bmatrix}$, where $\bm{f}(\cdot)$ is used in this proof in the context of \cite{Robinson1980}. By our hypotheses, we have that the hypotheses of Theorem 2.1 of \cite{Robinson1980} are satisfied. To apply Corollary 2.2 from \cite{Robinson1980}, it remains to show that there exists a constant $\nu$ such that
\begin{align}
\| \bm{f} (\bm{u}_1, \bm{\lambda}, \bm{\mu}, \bm{\alpha}) - \bm{f} (\bm{u}_2, \bm{\lambda}, \bm{\mu}, \bm{\alpha}) \| \leq \nu \| \bm{u}_1 - \bm{u}_2 \|, \label{eq:EMStability.0}
\end{align}
for all $\bm{u}_1, \bm{u}_2 \in \mathcal{N}$ and $(\bm{\lambda}, \bm{\mu}, \bm{\alpha}) \in V$, where $V$ is a closed neighborhood of $(\bm{\lambda}^*, \bm{\mu}^*, \bm{\alpha}^*)$. To this end, observe that
\begin{align}
\nabla_{\lambda} \mathcal{L} (\bm{u}; \bm{\lambda}, \bm{\mu}, \bm{\alpha})^{\intercal}
&= \nabla \bm{h} (\bm{u}) \nabla f (\bm{u})^{\intercal} + \nabla \bm{h} (\bm{u}) \nabla \bm{h} (\bm{u})^{\intercal} \bm{\lambda} + \nabla \bm{h} (\bm{u}) \nabla \bm{r} (\bm{u})^{\intercal} \bm{\mu} + 2 \gamma \bm{\lambda}. \label{eq:EMStability.1}
\end{align}
Hence,
\begin{align}
\| \nabla_{\lambda} \mathcal{L} (\bm{u}_1; \bm{\lambda}, \bm{\mu}, \bm{\alpha})^{\intercal} - \nabla_{\lambda} \mathcal{L} (\bm{u}_2; \bm{\lambda}, \bm{\mu}, \bm{\alpha})^{\intercal} \|
&\leq \| \nabla \bm{h} (\bm{u}_1) \nabla f (\bm{u}_1)^{\intercal} - \nabla \bm{h} (\bm{u}_2) \nabla f (\bm{u}_2)^{\intercal} \| \nonumber \\
&\quad + \| \nabla \bm{h} (\bm{u}_1) \nabla \bm{h} (\bm{u}_1)^{\intercal} - \nabla \bm{h} (\bm{u}_2) \nabla \bm{h} (\bm{u}_2)^{\intercal} \| \| \bm{\lambda} \| \nonumber \\
&\quad + \| \nabla \bm{h} (\bm{u}_1) \nabla \bm{r} (\bm{u}_1)^{\intercal} - \nabla \bm{h} (\bm{u}_2) \nabla \bm{r} (\bm{u}_2)^{\intercal} \| \| \bm{\mu} \|. \label{eq:EMStability.2}
\end{align}
As $(\bm{\lambda}, \bm{\mu}, \bm{\alpha})$ is in a closed neighborhood $V$ of $(\bm{\lambda}^*, \bm{\mu}^*, \bm{\alpha}^*)$, there exist constants $c_1$ and $c_2$ such that $\| \bm{\lambda} \| \leq c_1$ and $\| \bm{\mu} \| \leq c_2$ for all $(\bm{\lambda}, \bm{\mu}, \bm{\alpha}) \in V$. As $f, \bm{h}, \bm{r} \in \mathcal{C}^1$, the functions $\nabla \bm{h} (\cdot) \nabla f (\cdot)^{\intercal}$, $\nabla \bm{h} (\cdot) \nabla \bm{h} (\cdot)^{\intercal}$, and $\nabla \bm{h} (\cdot) \nabla \bm{r} (\cdot)^{\intercal}$ are Lipschitz continuous on $\mathcal{N}$. Combining these facts with (\ref{eq:EMStability.2}) yields that there exists a constant $\nu$ such that
\begin{align}
\| \nabla_{\lambda} \mathcal{L} (\bm{u}_1; \bm{\lambda}, \bm{\mu}, \bm{\alpha})^{\intercal} - \nabla_{\lambda} \mathcal{L} (\bm{u}_2; \bm{\lambda}, \bm{\mu}, \bm{\alpha})^{\intercal} \|
&\leq \nu \| \bm{u}_1 - \bm{u}_2 \|, \label{eq:EMStability.3}
\end{align}
for all $\bm{u}_1, \bm{u}_2 \in \mathcal{N}$ and $(\bm{\lambda}, \bm{\mu}, \bm{\alpha}) \in V$. Next, observe that
\begin{align}
\nabla_{\mu} \mathcal{L} (\bm{u}; \bm{\lambda}, \bm{\mu}, \bm{\alpha})^{\intercal}
&= \nabla \bm{r} (\bm{u}) \nabla f (\bm{u})^{\intercal} + \nabla \bm{r} (\bm{u}) \nabla \bm{h} (\bm{u})^{\intercal} \bm{\lambda} + \nabla \bm{r} (\bm{u}) \nabla \bm{r} (\bm{u})^{\intercal} \bm{\mu} - \bm{r} (\bm{u}) - \bm{\alpha} + 2 \gamma \bm{\mu}. \label{eq:EMStability.4}
\end{align}
Hence,
\begin{align}
\| \nabla_{\mu} \mathcal{L} (\bm{u}_1; \bm{\lambda}, \bm{\mu}, \bm{\alpha})^{\intercal} - \nabla_{\mu} \mathcal{L} (\bm{u}_2; \bm{\lambda}, \bm{\mu}, \bm{\alpha})^{\intercal} \|
&\leq \| \nabla \bm{r} (\bm{u}_1) \nabla f (\bm{u}_1)^{\intercal} - \nabla \bm{r} (\bm{u}_2) \nabla f (\bm{u}_2)^{\intercal} \| \nonumber \\
&\quad + \| \nabla \bm{r} (\bm{u}_1) \nabla \bm{h} (\bm{u}_1)^{\intercal} - \nabla \bm{r} (\bm{u}_2) \nabla \bm{h} (\bm{u}_2)^{\intercal} \| \| \bm{\lambda} \| \nonumber \\
&\quad + \| \nabla \bm{r} (\bm{u}_1) \nabla \bm{r} (\bm{u}_1)^{\intercal} - \nabla \bm{r} (\bm{u}_2) \nabla \bm{r} (\bm{u}_2)^{\intercal} \| \| \bm{\mu} \| \nonumber \\
&\quad + \| \bm{r} (\bm{u}_1) - \bm{r} (\bm{u}_2) \|. \label{eq:EMStability.5}
\end{align}
By an argument similar to the one establishing (\ref{eq:EMStability.3}), there exists a constant $\nu$ such that
\begin{align}
\| \nabla_{\mu} \mathcal{L} (\bm{u}_1; \bm{\lambda}, \bm{\mu}, \bm{\alpha})^{\intercal} - \nabla_{\mu} \mathcal{L} (\bm{u}_2; \bm{\lambda}, \bm{\mu}, \bm{\alpha})^{\intercal} \|
&\leq \nu \| \bm{u}_1 - \bm{u}_2 \|, \label{eq:EMStability.6}
\end{align}
for all $\bm{u}_1, \bm{u}_2 \in \mathcal{N}$ and $(\bm{\lambda}, \bm{\mu}, \bm{\alpha}) \in V$. Now combining (\ref{eq:EMStability.3}) and (\ref{eq:EMStability.6}) yields (\ref{eq:EMStability.0}). Corollary 2.2 from \cite{Robinson1980} can now be applied to conclude the proof.
\end{proof}

Note that the inclusion of $\gamma \left\| [\bm{\lambda}, \bm{\mu}] \right\|^2$ in the objective function of problem (\ref{prob:Mu-gamma}) is to ensure positive definiteness of the hessian. If we instead considered the problem
\begin{align}
\begin{array}{cc}
\displaystyle \min_{\bm{\lambda}, \bm{\mu}} & E_{m,0} (\bm{u}, \bm{\lambda}, \bm{\mu}) \\
\text{s.t.} & \bm{\mu} \geq \bm{0}.
\end{array} \label{prob:Mu}
\end{align}
then achieving a result analogous to Lemma~\ref{lem:AltLem13.2} for problem (\ref{prob:Mu}) would require the additional, and much stronger, hypothesis that {\footnotesize $\displaystyle \begin{bmatrix} \nabla \bm{h} (\bm{x}^*) \\ \bm{A} \end{bmatrix}$} is of full row rank.

Before moving to the convergence analysis for the multiplier step, we provide a remark on what can be said on the stability of problem (\ref{prob:Mu-gamma}) if $E_{m,0}$ is replaced with $E_{m,1}$. To this end, we note that $E_{m,1} (\bm{u}, \bm{\lambda}, \bm{\mu})$ must be twice continuously differentiable with respect to $\bm{\lambda}$ and $\bm{\mu}$ in a neighborhood of a solution $(\bm{\lambda}^*, \bm{\mu}^*)$ to establish a satibility result similar to Lemma~\ref{lem:AltLem13.2}. We can guarantee this holds provided that the hypotheses of Lemma~\ref{lem:E1diff} are satisfied, that is, in addition to the hypotheses of Lemma~\ref{lem:AltLem13.2} we require that assumption (\textbf{SCS}) holds at $(\bm{x}^*, \bm{\lambda}^*, \bm{\mu}^*)$. This result is much more restrictive as it would limit future analysis to only nondegenerate stationary points of problem (\ref{prob:main-nlp}). As such, this provides some motivation behind the choice of minimizing $E_{m,0}$ with respect to $\bm{\lambda}$ and $\bm{\mu}$ instead of $E_{m,1}$ in the Multiplier Step of NPASA. We note that the additional minization of $E_{m,1}$ with respect to $\bm{\mu}$ in the Multiplier Step is required in order to ensure local quadratic convergence of the multiplier step in the LS algorithm as will be illustrated in the proof of Theorem~\ref{ConvergenceResult}.

With these lemmas in place, we are now ready to state and prove a convergence result for the multiplier step of NPASA, found in Algorithm~\ref{alg:cms}. 
%TODO
%This result is a modified version of Theorem 4.1 in \cite{Hager1993} accounting for the updates made to the algorithm. While there are many required updates to the proof, we note that some details of the proof are essentially the same as in the proof of Theorem 4.1 in \cite{Hager1993} but are included here for completeness.

\begin{theorem} \label{ConvergenceResult}
Suppose that $f, \bm{h} \in \mathcal{C}^2$ and that either assumptions (\textbf{LICQ}), (\textbf{SCS}), and (\textbf{SOSC}) are satisfied at $(\bm{x}^*, \bm{\lambda}^*, \bm{\mu}^*)$ or assumptions (\textbf{LICQ}) and (\textbf{SSOSC}) are satisfied at $(\bm{x}^*, \bm{\lambda}^*, \bm{\mu}^*)$. %If $p_k$ is chosen such that $p_k \geq 1$ and $\frac{1}{p_k} \leq \| h (\bm{x}_k) \|^2$ for all $k \geq 0$ then 
Then there exists a neighborhood $\mathcal{N}_1$ and $\mathcal{N}_2$ of $\bm{x}^*$ and a constant $c$ such that, for every $\bm{z}_0 \in \mathcal{N}_1 \cap \Omega$, the primal variables and equality multipliers, $(\bm{z}, \bm{\nu})$, generated using the scheme
\begin{align}
%(\bm{\nu}_i, \bm{\eta}_i) &\in \arg\min \left\{ E_{m,0} ( \bm{z}_i, \bm{\nu}, \bm{\eta}) : \bm{\eta} \geq \bm{0} \right\} \label{scheme:MS1} \\
(\bm{\nu}_i, \bm{\eta}_i) &\in \arg\min \left\{ E_{m,0} ( \bm{z}_i, \bm{\nu}, \bm{\eta}) + \gamma \left\| [\bm{\lambda}, \bm{\mu}] \right\|^2 : \bm{\eta} \geq \bm{0} \right\} \label{scheme:MS1} \\
\bm{\eta}_i' &\in \arg\min \left\{ E_{m,1} ( \bm{z}_i, \bm{\nu}_i, \bm{\eta}) : \bm{\eta} \geq \bm{0} \right\} \label{scheme:MS2} \\
\bm{z}_{i+1} &= \arg \min \left\{ \mathcal{L}_{p_i}^i (\bm{z}, \bm{\nu}_i) : \nabla \bm{h} (\bm{z}_i) (\bm{z} - \bm{z}_i) = 0, \bm{z} \in \mathcal{N}_2 \cap \Omega \right\} \label{scheme:MS3}
\end{align}
%in the multiplier step of NPASA 
converge to a point $( \bm{z}^*, \bm{\nu}^* )$ and
\iffalse
\begin{itemize}
\item[(\ref{ConvergenceResult}.0)] $\displaystyle \min_{\bm{\eta} \geq \bm{0}} E_{m,1} ( \bm{z}^*, \bm{\nu}^*, \bm{\eta} ) = 0$.
\end{itemize}
\fi
\begin{align}
\min_{\bm{\eta} \geq \bm{0}} E_{m,1} ( \bm{z}^*, \bm{\nu}^*, \bm{\eta} ) = 0. \label{ConvergenceResult.0}
\end{align}
Furthermore, the following properties hold for $i \geq 0$:
\iffalse
\begin{itemize}
\item[(\ref{ConvergenceResult}.1)] $\| \bm{z}_i - \bm{z}^* \| \leq c 2^{-2^i}$.
\item[(\ref{ConvergenceResult}.2)] $\| \bm{z}_{i+1} - \bm{z}_1 \| \leq c \| \bm{z}_0 - \bm{x}^* \|^2$.
\item[(\ref{ConvergenceResult}.3)] $\| \bm{z}_{i+1} - \bm{x}^* \| \leq c \| \bm{z}_0 - \bm{x}^* \|^2 + c \| h (\bm{z}_0) \|$.
\item[(\ref{ConvergenceResult}.4)] $E_{m,1} \left( \bm{z}_{i+1}, \bm{\nu}_{i+1}, \bm{\eta}_{i+1}' \right) \leq c E_{m,1} \left( \bm{z}_i, \bm{\nu}_i, \bm{\eta}_i' \right)^2$.
\item[(\ref{ConvergenceResult}.5)] $E_1 \left( \bm{z}_i, \bm{\nu}_i, \bm{\eta}_i' \right) \leq \| \bm{z}_i - \bm{x}^* \|$.
%\item[(\ref{ConvergenceResult}.6)] $\displaystyle \lim_{i \to \infty} E_{m,1} \left( \bm{z}_i, \bm{\nu}_i, \bm{\eta}_i' \right) = 0$.
\end{itemize}
\fi
\begin{align}
&\| \bm{z}_i - \bm{z}^* \| \leq c 2^{-2^i}, \label{ConvergenceResult.1} \\
&\| \bm{z}_{i+1} - \bm{z}_1 \| \leq c \| \bm{z}_0 - \bm{x}^* \|^2, \label{ConvergenceResult.2} \\
&\| \bm{z}_{i+1} - \bm{x}^* \| \leq c \| \bm{z}_0 - \bm{x}^* \|^2 + c \| h (\bm{z}_0) \|, \label{ConvergenceResult.3} \\
&E_{m,1} \left( \bm{z}_{i+1}, \bm{\nu}_{i+1}, \bm{\eta}_{i+1}' \right) \leq c E_{m,1} \left( \bm{z}_i, \bm{\nu}_i, \bm{\eta}_i' \right)^2, \label{ConvergenceResult.4} \\
&E_1 \left( \bm{z}_i, \bm{\nu}_i, \bm{\eta}_i' \right) \leq \| \bm{z}_i - \bm{x}^* \|. \label{ConvergenceResult.5}
%\lim_{i \to \infty} E_{m,1} \left( \bm{z}_i, \bm{\nu}_i, \bm{\eta}_i' \right) &= 0. \label{ConvergenceResult.6}
\end{align}
\end{theorem}

Before providing the proof for this theorem, we note that the iterative scheme in (\ref{scheme:MS1}) -- (\ref{scheme:MS3}) is the same update provided in the LS pseudocode, Algorithm~\ref{alg:cms}, without checking for a sufficient decrease in the error estimator at each iteration. %that would result in reusing the same multipliers at the next iteration. 
As a result, it follows from (\ref{ConvergenceResult.4}) that the criterion $E_{m,1} \left( \bm{z}_{i+1}, \bm{\nu}_{i+1}, \bm{\eta}_{i+1}' \right) \leq \delta E_{m,1} \left( \bm{z}_i, \bm{\nu}_i, \bm{\eta}_i' \right)$ in the Multiplier Step of the LS algorithm will always be satisfied under the hypotheses of Theorem~\ref{ConvergenceResult} for $i$ sufficiently large. Hence, this local convergence result can be applied to the Multiplier step of the LS algorithm without modification. We now provide a proof for Theorem~\ref{ConvergenceResult}.

\begin{proof}[Proof of Theorem~\ref{ConvergenceResult}]
We will perform our analysis of this method by considering the perturbed problem (\ref{prob:Ru}). As problem (\ref{prob:main-nlp}) is equivalent to problem (\ref{prob:Ru}) with $\bm{u} = \bm{0}$, we have that $\bm{x}^*$ is a solution of problem (\ref{prob:Ru}) with $\bm{u} = \bm{0}$. Additionally, we will let $\bm{z}_i^*$ denote the local minimizer of (\ref{prob:Ru}) with $\bm{u} = \bm{h}(\bm{z}_i)$ and $( \bm{\nu}_i^*, \bm{\eta}_i^* )$ be corresponding multipliers for which $E_{m,1} \left( \bm{z}_i^*, \bm{\nu}_i^*, \bm{\eta}_i^* \right) = 0$. To justify the existence of $( \bm{\nu}_i^*, \bm{\eta}_i^* )$, observe that the Lagrangian for problem (\ref{prob:Ru}) with $\bm{u} = \bm{h}(\bm{z}_i)$ is 
\begin{align}
f(\bm{x}) + \bm{\lambda}^{\intercal} \left( \bm{h}(\bm{x}) - \bm{h}(\bm{z}_i) \right) + \bm{\mu}^{\intercal} \bm{r} (\bm{x})
\end{align} 
which has gradient 
\begin{align}
\nabla f(\bm{x}) + \bm{\lambda}^{\intercal} \nabla \bm{h}(\bm{x}) + \bm{\mu}^{\intercal} \nabla \bm{r} (\bm{x}).
\end{align}
As this gradient is the same as the gradient of the Lagrangian for problem (\ref{prob:main-nlp}), we have that a KKT point $\left( \bm{z}_i^*, \bm{\nu}_i^*, \bm{\eta}_i^* \right)$ for problem (\ref{prob:Ru}) satisfies $E_{m,1} \left( \bm{z}_i^*, \bm{\nu}_i^*, \bm{\eta}_i^* \right) = 0$, as desired.

\iffalse % No longer necessary for analysis (keeping for now in case it comes up later)
Finally, we define 
\begin{align*}
\bm{y}_k^* = - \bm{h}(\bm{x}_k) - \nabla \bm{h}(\bm{x}_k) (\bm{x}_k^* - \bm{x}_k).
\end{align*}
\fi

Suppose that $\bm{z}_i$ is near $\bm{x}^*$. %By Lemma 13.1 in \cite{Hager1993} with $\bm{u} = \bm{h} (\bm{x}_k)$, we have that
By (\ref{lem:Lem13.1.1}) in Lemma~\ref{lem:Lem13.1}, we have $\bm{x} (\bm{h} (\bm{z}_i)) = \bm{z}_i^*$ and $\bm{x} (\bm{h} (\bm{x}^*)) = \bm{x}^*$. Hence, using (\ref{lem:Lem13.1.3}) in Lemma~\ref{lem:Lem13.1} with $\bm{u}_1 = \bm{h} (\bm{z}_i)$ and $\bm{u}_2 = \bm{h} (\bm{x}^*) = \bm{0}$, it follows that
\begin{align}
\| \bm{z}_i^* - \bm{x}^* \| \leq c \| \bm{h}( \bm{z}_i ) - \bm{h}(\bm{x}^*) \| = c \| \bm{h}( \bm{z}_i ) \|.
\end{align} 
Since $\| \bm{h}(\bm{z}_i) \| \to 0$ as $\bm{z}_i \to \bm{x}^*$, we have that $\bm{z}_i^*$ is near $\bm{x}^*$ when $\bm{z}_i$ is near $\bm{x}^*$. As $\bm{x} = \bm{z}_i$ satisfies the constraint $\bm{h} (\bm{x}) = \bm{u}$ when $\bm{u} = \bm{h} (\bm{z}_i)$, it follows from Theorem~\ref{Theorem3.1} that
\begin{align}
\| \bm{z}_{i+1} - \bm{z}_i^* \| &\leq c \| \bm{\nu}_i - \bm{\nu}_i^* \|^2 + c \| \bm{z}_i - \bm{z}_i^* \|^2 + c \| \bm{h} (\bm{z}_i) - \bm{h} (\bm{z}_i) \| \\
&= c \| \bm{\nu}_i - \bm{\nu}_i^* \|^2 + c \| \bm{z}_i - \bm{z}_i^* \|^2 \\
&\leq c \left\| \begin{bmatrix} \bm{\nu}_i, \ \bm{\eta}_i \end{bmatrix} - \begin{bmatrix} \bm{\nu}_i^*, \ \bm{\eta}_i^* \end{bmatrix} \right\|^2 + c \| \bm{z}_i - \bm{z}_i^* \|^2, \label{NPASA8.1}
\end{align}
for some constant $c$. As a note of interest to the reader, observe that the penalized Lagrangian function for problem (\ref{prob:Ru}) is given by
\begin{align}
\mathcal{L}_{p} (\bm{\nu}, \bm{z}) = f(\bm{z}) + \bm{\nu}^{\intercal} \bm{h} (\bm{z}) + p \| \bm{h} (\bm{z}) - \bm{u} \|^2.
\end{align}
Hence, when we take $\bm{u} = \bm{h} (\bm{z}_i)$ we have that the penalized Lagrangian function for problem (\ref{prob:Ru}) is equal to $\mathcal{L}_{p}^i (\bm{\nu}, \bm{z})$.
%
\iffalse % No longer needed since corollary covers this part
Next, observe that
\begin{align}
\| \bm{z}_k - \bm{x}_k^* \|^2 &= \| \bm{z}_k - \bm{x}_k + \bm{x}_k - \bm{x}_k^* \|^2 \nonumber \\
&\leq \| \bm{z}_k - \bm{x}_k \|^2 + 2 \left| \langle \bm{z}_k - \bm{x}_k, \bm{x}_k - \bm{x}_k^* \rangle \right| + \| \bm{x}_k - \bm{x}_k^* \|^2 \nonumber \\
&\leq 2 \| \bm{z}_k - \bm{x}_k \|^2 + 2 \| \bm{x}_k - \bm{x}_k^* \|^2 \label{NPASA8.1.1} \\
&\leq 2 \left( c_1 \| h (\bm{x}_k) \| \right)^2 + 2 \| \bm{x}_k - \bm{x}_k^* \|^2 \label{NPASA8.1.2} \\
&= 2 c_1^2 \| h (\bm{x}_k) \|^2 + 2 \| \bm{x}_k - \bm{x}_k^* \|^2, \label{NPASA8.2}
\end{align}
where (\ref{NPASA8.1.1}) follows from Cauchy and (\ref{NPASA8.1.2}) follows from Theorem~\ref{NewtonYSchemeThm} with $\bm{z}_0 = \bm{x}_k$. Combining (\ref{NPASA8.1}) and (\ref{NPASA8.2}) yields
\begin{align}
\| \bm{x}_{k+1} - \bm{x}_k^* \| \leq c \| [\bm{\lambda}_k \ \bm{\mu}_k] - [\bm{\lambda}_k^* \ \bm{\mu}_k^*] \|^2 + c \| \bm{x}_k - \bm{x}_k^* \|^2 + c \| h (\bm{x}_k) \|^2. \label{NPASA8.2.1}
\end{align}
\fi
%
%Next, by either assumptions (\textbf{LICQ}), (\textbf{SCS}), and (\textbf{SOSC}) or assumptions (\textbf{LICQ}) and (\textbf{SSOSC}) we have that the hypotheses of Lemma~\ref{lem:Lem13.2} are satisfied. Hence, 
As the hypotheses of Lemma~\ref{lem:AltLem13.2} are satisfied,
it follows from (\ref{lem:AltLem13.2.3}) in Lemma~\ref{lem:AltLem13.2} with $\bm{u}_1 = \bm{z}_i$ and $\bm{u}_2 = \bm{z}_i^*$ that 
\begin{align}
\left\| \begin{bmatrix} \bm{\nu}_i, \ \bm{\eta}_i \end{bmatrix} - \begin{bmatrix} \bm{\nu}_i^*, \ \bm{\eta}_i^* \end{bmatrix} \right\| = O \left( \| \bm{z}_i - \bm{z}_i^* \| \right). \label{NPASA8.3}
\end{align}
Similarly, $\left\| \begin{bmatrix} \bm{\nu}_{i+1}, \ \bm{\eta}_{i+1} \end{bmatrix} - \begin{bmatrix} \bm{\nu}_i^*, \ \bm{\eta}_i^* \end{bmatrix} \right\| = O \left( \| \bm{z}_{i+1} - \bm{z}_i^* \| \right)$ follows from an additional application of (\ref{lem:AltLem13.2.3}) in Lemma~\ref{lem:AltLem13.2}. Now (\ref{NPASA8.1}) and (\ref{NPASA8.3}) together imply that
\begin{align}
\| \bm{z}_{i+1} - \bm{z}_i^* \| + \left\| \begin{bmatrix} \bm{\nu}_{i+1}, \ \bm{\eta}_{i+1} \end{bmatrix} - \begin{bmatrix} \bm{\nu}_i^*, \ \bm{\eta}_i^* \end{bmatrix} \right\| \leq c \| \bm{z}_i - \bm{z}_i^* \|^2. \label{NPASA8.3.1}
\end{align}

By our choice of Problem (\ref{prob:Ru}), for $\bm{u} = \bm{h}(\bm{z}_i)$ we have that $\bm{h}(\bm{z}_i^*) = \bm{h}(\bm{z}_i)$. Hence,
\begin{align}
\| \bm{h}(\bm{z}_{i+1}) - \bm{h}(\bm{z}_i) \| &= \| \bm{h}(\bm{z}_{i+1}) - \bm{h}(\bm{z}_i^*) \| \leq c \| \bm{z}_{i+1} - \bm{z}_i^* \| \leq c \| \bm{z}_i - \bm{z}_i^* \|^2, \label{NPASA8.10}
\end{align}
where (\ref{NPASA8.10}) follows from performing a first order Taylor expansion of $\bm{h}$ and using (\ref{NPASA8.3.1}). Next, from (\ref{lem:Lem13.1.1}) in Lemma~\ref{lem:Lem13.1} we have that $\bm{x} (\bm{h} (\bm{z}_{i+1})) = \bm{z}_{i+1}^*$ and $\bm{x} (\bm{h} (\bm{z}_i)) = \bm{z}_i^*$. Hence, using (\ref{lem:Lem13.1.3}) in Lemma~\ref{lem:Lem13.1} with $\bm{u}_1 = \bm{h} (\bm{z}_{i+1})$ and $\bm{u}_2 = \bm{h} (\bm{z}_i)$ together with (\ref{NPASA8.10}) yields
\begin{align}
\| \bm{z}_{i+1}^* - \bm{z}_i^* \| &\leq c \| \bm{h}(\bm{z}_{i+1}) - \bm{h}(\bm{z}_i) \| \leq c \| \bm{z}_i - \bm{z}_i^* \|^2. \label{NPASA8.11}
\end{align}
Now define
\begin{align}
r_{i+1} := \| \bm{z}_{i+1} - \bm{z}_i^* \| + \| \bm{z}_{i+1}^* - \bm{z}_i^* \|. \label{NPASA8.12}
\end{align}
We claim that $r_{i+1} = O (r_i^2)$. To prove this claim, first note that combining (\ref{NPASA8.3.1}), (\ref{NPASA8.11}), and (\ref{NPASA8.12}) yields
\begin{align}
r_{i+1} \leq c \| \bm{z}_i - \bm{z}_i^* \|. \label{NPASA8.12.1}
\end{align}
By Cauchy, we have
\begin{align}
\| \bm{z}_i - \bm{z}_i^* \| \leq 2 \| \bm{z}_i - \bm{z}_{i-1}^* \|^2 + 2 \| \bm{z}_i^* - \bm{z}_{i-1}^* \|^2 \label{NPASA8.13}
\end{align}
which combined with (\ref{NPASA8.12.1}) yields $r_{i+1} = O (r_i^2)$, as claimed.
%%%%%%%%%%%%%%%%%%%%%%%%%%%%%%%%%%%%%%%%%%%%%%%
% No longer needed since this term doesn't show up in bound but might be useful later
\iffalse 
Additionally, performing a second order Taylor expansion and applying $\nabla \bm{h} (\bm{z}_k) (\bm{x}_{k+1} - \bm{z}_k) = 0$ from the constraint in (A2) yields that
\begin{align}
\| h (\bm{x}_{k+1}) \| &\leq \| h (\bm{z}_k) \| + c \| \bm{x}_{k+1} - \bm{z}_k \|^2 \leq \| h (\bm{z}_k) \| + c \| \bm{x}_{k+1} - \bm{x}_k^* \|^2 + c \| \bm{z}_k - \bm{x}_k^* \|^2 \label{NPASA8.14}
\end{align}
As $\| h (\bm{z}_k) \| \leq c \| h (\bm{x}_k) \|^2$ from (\ref{NewtonYSchemeThm.3}) in Theorem~\ref{NewtonYSchemeThm}, substituting the bounds established in (\ref{NPASA8.2}) and (\ref{NPASA8.3.1}) into (\ref{NPASA8.14}) yields
\begin{align}
\| h (\bm{x}_{k+1}) \| &\leq c \| h (\bm{x}_k) \|^2 + c \| \bm{x}_k - \bm{x}_k^* \|^2 + c \| h (\bm{x}_k) \|^4 + c \| \bm{x}_k - \bm{x}_k^* \|^4. \nonumber
\end{align}
Hence, 
\begin{align}
\| h (\bm{x}_{k+1}) \| = O (\| h (\bm{x}_k) \|^2) + O (\| \bm{x}_k^* - \bm{x}_k \|^2). \label{NPASA8.15}
\end{align}
\fi
%%%%%%%%%%%%%%%%%%%%%%%%%%%%%%%%%%%%%%%%%%%%%%%
So there exists a constant $\gamma$ such that $r_{i+1} \leq \gamma r_i^2$. Then we have
\begin{align}
r_{i+1} 
&\leq \frac{1}{\gamma} \left( \gamma r_i \right)^2 \leq \ldots \leq \frac{1}{\gamma} \left( \gamma r_1 \right)^{2^i} = r_1 \left( \gamma r_1 \right)^{2^i - 1}.
\end{align}
By applying (\ref{NPASA8.3.1}) and (\ref{NPASA8.11}) together with the triangle and Cauchy inequalities we have
\begin{align}
r_1
&= \| \bm{z}_1 - \bm{z}_0^* \| + \| \bm{z}_1^* - \bm{z}_0^* \| %\nonumber \\
\leq c \| \bm{z}_0 - \bm{z}_0^* \|^2 %\tag{Applying \ref{NPASA8.3.1} and \ref{NPASA8.11}} \\
\leq c \| \bm{z}_0 - \bm{x}^* \|^2 + c \| \bm{z}_0^* - \bm{x}^* \|^2. %\tag{Applying triangle inequality and Cauchy}
\end{align}
Using (\ref{lem:Lem13.1.3}) in Lemma~\ref{lem:Lem13.1} with $\bm{u}_1 = \bm{h} (\bm{z}_0)$ and $\bm{u}_2 = \bm{h} (\bm{x}^*)$ it follows that
\begin{align}
\| \bm{z}_0^* - \bm{x}^* \|
&\leq c \| \bm{h} (\bm{z}_0) - \bm{h} (\bm{x}^*) \| 
\leq c \| \bm{z}_0 - \bm{x}^* \|
\end{align}
which combined with the bound on $r_1$ yields
\begin{align}
r_1 &\leq c \| \bm{z}_0 - \bm{x}^* \|^2. \label{NPASA8.14.4}
\end{align}
It now follows that $r_1 \to 0$ as $\bm{z}_0 \to \bm{x}^*$. Thus, there exists a constant $c$ such that $r_i \leq c r_1 2^{-2^i - 1}$, for all $\bm{z}_0$ sufficiently close to $\bm{x}^*$. It now follows that the sequence $\{ \bm{z}_i \}_{i = 0}^{\infty}$ is Cauchy since
\begin{align}
\| \bm{z}_{i+1} - \bm{z}_i \| 
&\leq \| \bm{z}_{i+1} - \bm{z}_i^* \| + \| \bm{z}_i^* - \bm{z}_{i-1}^* \| + \| \bm{z}_i - \bm{z}_{i-1}^* \| %\\
\leq r_{i+1} + r_i %\\
\leq c r_1 2^{-2^i}, \label{NPASA8.15}
\end{align}
for all $i \geq 1$. Thus, $\{ \bm{z}_i \}_{i = 0}^{\infty}$ has a limit, say $\bm{z}^*$, and it follows from (\ref{NPASA8.15}) that $\bm{z}^*$ satisfies
\begin{align}
\| \bm{z}_{i} - \bm{z}^* \| &\leq c 2^{-2^i},
\end{align}
for all $i \geq 0$, hence, establishing (\ref{ConvergenceResult.1}). In particular, note that (\ref{NPASA8.15}) also implies that $\{ \bm{z}_i^* \}_{i = 0}^{\infty}$ converges to $\bm{z}^*$.

Now we claim that $\{ \left[ \bm{\nu}_i, \bm{\eta}_i \right] \}_{i = 0}^{\infty}$ converges to a limit. Let $\varepsilon > 0$ be given. As $\bm{z}_i \to \bm{z}^*$ and $\bm{z}_i^* \to \bm{z}^*$, there exists a $M \geq 0$ such that $\| \bm{z}_i - \bm{z}_i^* \|^2 < \varepsilon$, for all $i \geq M$. Hence, 
\begin{align}
\| [\bm{\nu}_{i+1} \ \bm{\eta}_{i+1}] - [\bm{\nu}_i \ \bm{\eta}_i] \| 
&\leq \| [\bm{\nu}_{i+1} \ \bm{\eta}_{i+1}] - [\bm{\nu}_i^* \ \bm{\eta}_i^*] \| + \| [\bm{\nu}_i \ \bm{\eta}_i] - [\bm{\nu}_i^* \ \bm{\eta}_i^*] \| \\
&\leq c \| \bm{z}_i - \bm{z}_i^* \|^2 + c \| \bm{z}_i - \bm{z}_i^* \|^2 \\
&\leq 2c \varepsilon,
\end{align}
for all $i \geq M$. So $\{ \left[ \bm{\nu}_i, \bm{\eta}_i \right] \}_{i = 0}^{\infty}$ is Cauchy and therefore has a limit, say $\left[ \bm{\nu}^*, \bm{\eta}^* \right]$. Similarly, it follows that $\{ \left[ \bm{\nu}_i^*, \bm{\eta}_i^* \right] \}_{i = 0}^{\infty}$ converges to $\left[ \bm{\nu}^*, \bm{\eta}^* \right]$. As $E_{m,1}$ is continuous, we now have that
\begin{align}
\min_{\bm{\eta} \geq \bm{0}} E_{m,1} (\bm{z}^*, \bm{\nu}^*, \bm{\eta}) 
&\leq E_{m,1} (\bm{z}^*, \bm{\nu}^*, \bm{\eta}^*) 
= \lim_{i \to \infty} E_{m,1} (\bm{z}_i^*, \bm{\nu}_i^*, \bm{\eta}_i^*). \label{NPASA8.15.1}
\end{align}
Since $E_{m,1} (\bm{z}_i^*, \bm{\nu}_i^*, \bm{\eta}_i^*) = 0$ for all $i \geq 0$, it follows that %$E_{m,1} (\bm{z}^*, \bm{\nu}^*, \bm{\eta}^*) = 0$ 
from (\ref{NPASA8.15.1}) that (\ref{ConvergenceResult.0}) holds.

Recalling (\ref{NPASA8.3.1}) and (\ref{NPASA8.11}) it follows that
\begin{align}
r_1 &= \| \bm{z}_{1} - \bm{z}_0^* \| + \| \bm{z}_{1}^* - \bm{z}_0^* \| \leq c \| \bm{z}_{0} - \bm{z}_0^* \|^2 + c \| h (\bm{z}_0) \|^2. \label{NPASA8.16}
\end{align}
Using $\bm{h}(\bm{x}^*) = \bm{0}$ we have that
\begin{align}
\| \bm{h} (\bm{z}_0) \| 
= \| \bm{h}(\bm{z}_0) - \bm{h}(\bm{x}^*) \| 
= \| \nabla \bm{h}(\xi_0) (\bm{z}_0 - \bm{x}^*) \| 
\leq c \| \bm{z}_0 - \bm{x}^* \|, \label{NPASA8.17}
\end{align}
for some $\bm{\xi}_0$ between $\bm{z}_0$ and $\bm{x}^*$. Also, using (\ref{lem:Lem13.1.3}) in Lemma~\ref{lem:Lem13.1} with $\bm{u}_1 = \bm{h} (\bm{z}_0)$ and $\bm{u}_2 = \bm{h} (\bm{x}^*) = \bm{0}$ together with the fact that $\bm{h} (\bm{z}_0^*) = \bm{h} (\bm{z}_0)$ we have that
\begin{align}
\| \bm{z}_0^* - \bm{x}^* \| 
\leq c \| \bm{h}(\bm{z}_0^*) - \bm{h}(\bm{x}^*) \| 
= c \| \bm{h}(\bm{z}_0) - \bm{h}(\bm{x}^*) \| 
\leq c \| \bm{z}_0 - \bm{x}^* \|. \label{NPASA8.17.1}
\end{align}
Hence,
\begin{align}
\| \bm{z}_0 - \bm{z}_0^* \| 
\leq \| \bm{z}_0 - \bm{x}^* \| + \| \bm{z}_0^* - \bm{x}^* \| 
\leq c \| \bm{z}_0 - \bm{x}^* \| \label{NPASA8.17.2}
\end{align}
which, when substituted into (\ref{NPASA8.16}) with (\ref{NPASA8.17}), yields
\begin{align}
r_1 &\leq c \| \bm{z}_{0} - \bm{x}^* \|^2. \label{NPASA8.18}
\end{align}
Next, applying the triangle inequality and using (\ref{NPASA8.15}) and (\ref{NPASA8.17}) we find that
\begin{align}
\| \bm{z}_{i+1} - \bm{z}_1 \| 
&\leq \sum_{j = 1}^{i} \| \bm{z}_{j+1} - \bm{z}_{j} \| 
\leq c r_1 \sum_{j = 1}^{i} 2^{-2^i} 
\leq c r_1 
\leq c \| \bm{z}_{0} - \bm{x}^* \|^2,
\end{align}
which proves (\ref{ConvergenceResult.2}).

Next, from (\ref{NPASA8.17.1}) we have that 
\begin{align}
\| \bm{z}_0^* - \bm{x}^* \| 
\leq c \| \bm{h}(\bm{z}_0^*) - \bm{h}(\bm{x}^*) \| 
= c \| \bm{h}(\bm{z}_0) \|
\end{align}
and from (\ref{NPASA8.3.1}) and (\ref{NPASA8.17.2}) we have $\| \bm{z}_1 - \bm{z}_0^* \| \leq c \| \bm{z}_0 - \bm{x}^* \|^2$. Using these together with the triangle inequality and (\ref{ConvergenceResult.3}) we have
\begin{align}
\| \bm{z}_{i+1} - \bm{x}^* \| 
&\leq \| \bm{z}_{i+1} - \bm{z}_1 \| + \| \bm{z}_1 - \bm{z}_0^* \| + \| \bm{z}_0^* - \bm{x}^* \|
\leq c \| \bm{z}_0 - \bm{x}^* \|^2 + c \| \bm{h} (\bm{z}_0) \|,
\end{align}
concluding the proof of (\ref{ConvergenceResult.3}).

To prove the remaining parts of the theorem, we proceed by defining
\begin{align}
E_{\bm{u}} \left( \bm{z}, \bm{\nu}, \bm{\eta} \right) 
= \sqrt{ \| \nabla \mathcal{L} \left( \bm{z}, \bm{\nu}, \bm{\eta} \right) \|^2 + \| \bm{h}(\bm{z}) - \bm{u} \|^2 } \label{def:Eu}
\end{align}
where $\bm{u}$ is a fixed vector and we note that $E_{\bm{u}} \in \mathcal{C}^2$ in a neighborhood of $(\bm{x}^*, \bm{\lambda}^*, \bm{\mu}^*)$ by our hypotheses. First, we show that (\ref{ConvergenceResult.4}) holds. As $(\bm{z}_i^*, \bm{\nu}_i^*, \bm{\eta}_i^*)$ is sufficiently close to $(\bm{x}^*, \bm{\lambda}^*, \bm{\mu}^*)$, by fixing $\bm{\eta} = \bm{\eta}_i^*$ we can apply Lemma 13.2 in \cite{Hager1993} to $E_{\bm{u}} \left( \bm{z}, \bm{\nu}, \bm{\eta}_i^* \right)$ with $\bm{u} = \bm{h} (\bm{z}_i)$ to yield that
\begin{align}
\| \bm{z}_{i+1} - \bm{z}_i^* \| + \| \bm{\nu}_{i+1} - \bm{\nu}_i^* \| 
&\geq c_1 E_{\bm{u}} \left( \bm{z}_{i+1}, \bm{\nu}_{i+1}, \bm{\eta}_i^* \right) \\
&= c_1 \sqrt{ \| \nabla \mathcal{L} \left( \bm{z}_{i+1}, \bm{\nu}_{i+1}, \bm{\eta}_i^* \right) \|^2 + \| \bm{h}(\bm{z}_{i+1}) - \bm{h}(\bm{z}_i) \|^2 } \\
&\geq c_1 \| \nabla \mathcal{L} \left( \bm{z}_{i+1}, \bm{\nu}_{i+1}, \bm{\eta}_i^* \right) \|, \label{NPASA8.7}
\end{align}
for some constant $c_1$. From (\ref{lem:AltLem13.2.3}) in Lemma~\ref{lem:AltLem13.2} with $\bm{u}_1 = \bm{z}_{i+1}$ and $\bm{u}_2 = \bm{z}_i^*$ we have that
\begin{align}
\left\| \bm{\nu}_{i+1} - \bm{\nu}_i^* \right\| = O \left( \| \bm{z}_{i+1} - \bm{z}_i^* \| \right) \label{NPASA8.7.1}
\end{align}
which combined with (\ref{NPASA8.7}) yields
\begin{align}
\| \bm{z}_{i+1} - \bm{z}_i^* \| ^2
&\geq c_1 \| \nabla \mathcal{L} \left( \bm{z}_{i+1}, \bm{\nu}_{i+1}, \bm{\eta}_i^* \right) \|^2. \label{NPASA8.7.2}
\end{align}
Now recall the function $\bm{\Phi} : \mathbb{R}^m \times \mathbb{R}^m \to \mathbb{R}^m$ where the $j$th component of $\bm{\Phi}$ is defined to be $\Phi_j (\bm{a}, \bm{b}) = \min \{ a_j, b_j \}$. As $(\bm{z}_i^*, \bm{\nu}_i^*, \bm{\eta}_i^*)$ is a KKT point for problem (\ref{prob:Ru}), we have that $\eta_{ij}^* r_j (\bm{z}_i^*) = 0$, for $1 \leq j \leq m$. Now fix $j \in \{ 1, 2, \ldots, m \}$. Suppose $\eta_{ij}^* = 0$. As $\bm{r} (\bm{z}_{i+1}) \leq \bm{0}$, we have
\begin{align}
\Phi_j (\bm{\eta}_i^*, -\bm{r} (\bm{z}_{i+1}))^2
&= \min \{ \eta_{ij}^*, -r_j (\bm{z}_{i+1}) \}^2
= \min \{ 0, -r_j (\bm{z}_{i+1}) \}^2
= 0. \label{NPASA8.7.3}
\end{align}
On the other hand, suppose $r_j (\bm{z}_i^*) = 0$. As $\bm{r} (\bm{z}) = \bm{A} \bm{z} - \bm{b}$, we have that
\begin{align}
\Phi_j (\bm{\eta}_i^*, -\bm{r} (\bm{z}_{i+1}))^2
&= \min \{ \eta_{ij}^*, -r_j (\bm{z}_{i+1}) \}^2 \\
&\leq \left( r_j (\bm{z}_{i+1}) \right)^2 \\
&= \left( r_j (\bm{z}_{i+1}) - r_j (\bm{z}_i^*) \right)^2 \\
&\leq \| \bm{A} \|^2 \| \bm{z}_{i+1} - \bm{z}_i^* \|^2. \label{NPASA8.7.4}
\end{align}
Combining (\ref{NPASA8.7.3}) and (\ref{NPASA8.7.4}) we have that
\begin{align}
c_1 \| \bm{\Phi} (\bm{\eta}_i^*, -\bm{r} (\bm{z}_{i+1})) \|^2
&\leq \| \bm{z}_{i+1} - \bm{z}_i^* \|^2. \label{NPASA8.7.5}
\end{align}
Recalling that $E_{m,1} \left( \bm{z}, \bm{\nu}, \bm{\eta} \right) = \| \nabla \mathcal{L} \left( \bm{z}, \bm{\nu}, \bm{\eta} \right) \|^2 + \| \bm{\Phi} (\bm{\eta}, -\bm{r} (\bm{z})) \|^2$ it follows from (\ref{NPASA8.7.2}) and (\ref{NPASA8.7.5}) that 
\begin{align}
c_1 E_{m,1} \left( \bm{z}_{i+1}, \bm{\nu}_{i+1}, \bm{\eta}_i^* \right) 
\leq \| \bm{z}_{i+1} - \bm{z}_i^* \|^2. \label{NPASA8.7.6}
\end{align}
Recalling (\ref{scheme:MS2}) we have that $\bm{\eta}_{i+1}' \in \arg\min \left\{ E_{m,1} \left( \bm{z}_{i+1}, \bm{\nu}_{i+1}, \bm{\eta} \right) : \bm{\eta} \geq \bm{0} \right\}$. As $\bm{\eta}_{i+1}^* \geq \bm{0}$, it follows from (\ref{NPASA8.7.6}) and (\ref{scheme:MS2}) that
\begin{align}
c_1 E_{m,1} \left( \bm{z}_{i+1}, \bm{\nu}_{i+1}, \bm{\eta}_{i+1}' \right) 
&\leq c_1 E_{m,1} \left( \bm{z}_{i+1}, \bm{\nu}_{i+1}, \bm{\eta}_i^* \right)
\leq \| \bm{z}_{i+1} - \bm{z}_i^* \|^2. \label{NPASA8.7.8}
\end{align}
Next, applying Theorem~\ref{thm:E1-error-bound} to problem (\ref{prob:Ru}) with $\bm{u} = \bm{h} (\bm{z}_i)$ yields that
\begin{align}
\| \bm{z}_i - \bm{z}_i^* \|^2 
%&\leq c_2 \left( \| \nabla \mathcal{L} \left( \bm{z}_i, \bm{\nu}_i, \bm{\eta}_i' \right) \|^2 + \| \bm{h}(\bm{z}_i) - \bm{h}(\bm{z}_i) \|^2 + \| \bm{\Phi} (\bm{\eta}_i', -\bm{r} (\bm{z}_i)) \|^2 \right) \nonumber \\
&\leq c_2 \left( \| \nabla \mathcal{L} \left( \bm{z}_i, \bm{\nu}_i, \bm{\eta}_i' \right) \|^2 + \| \bm{\Phi} (\bm{\eta}_i', -\bm{r} (\bm{z}_i)) \|^2 \right) %\nonumber \\
= c_2 E_{m,1} \left( \bm{z}_i, \bm{\nu}_i, \bm{\eta}_i' \right). \label{NPASA8.7.9}
\end{align}
Now combining (\ref{NPASA8.3.1}), (\ref{NPASA8.7.8}), and (\ref{NPASA8.7.9}) yields
\begin{align}
E_{m,1} \left( \bm{z}_{i+1}, \bm{\nu}_{i+1}, \bm{\eta}_{i+1}' \right)
&\leq c E_{m,1} \left( \bm{z}_i, \bm{\nu}_i, \bm{\eta}_i' \right)^2, \label{NPASA8.7.10}
\end{align}
which proves (\ref{ConvergenceResult.4}).

Now we show that (\ref{ConvergenceResult.5}) holds. By definition of (\ref{def:Eu}) and the fact that $\bm{h} (\bm{x}^*) = \bm{0}$, note that $E_{\bm{u}} \left( \bm{z}, \bm{\nu}, \bm{\eta} \right) = \sqrt{\| \nabla \mathcal{L} \left( \bm{z}, \bm{\nu}, \bm{\eta} \right) \|^2 + \| \bm{h} (\bm{z}) \|^2}$ when $\bm{u} = \bm{h}(\bm{x}^*)$. So by fixing $\bm{\eta} = \bm{\mu}^*$ then applying Lemma 13.2 in \cite{Hager1993} to $E_{\bm{u}} \left( \bm{z}, \bm{\nu}, \bm{\mu}^* \right)$ with $\bm{u} = \bm{h}(\bm{x}^*)$ we have that there exists a constant $c_1$ such that
\begin{align}
\| \bm{z}_{i+1} - \bm{x}^* \| + \| \bm{\nu}_{i+1} - \bm{\lambda}^* \|
&\geq c_1 E_{\bm{u}} \left( \bm{z}_{i+1}, \bm{\nu}_{i+1}, \bm{\mu}^* \right) \\
&= c_1 \sqrt{\| \nabla \mathcal{L} \left( \bm{z}_{i+1}, \bm{\nu}_{i+1}, \bm{\mu}^* \right) \|^2 + \| \bm{h} (\bm{z}_{i+1}) \|^2}. \label{NPASA8.8}
\end{align}
From (\ref{lem:AltLem13.2.3}) in Lemma~\ref{lem:AltLem13.2} with $\bm{u}_1 = \bm{z}_{i+1}$ and $\bm{u}_2 = \bm{x}^*$ we have
\begin{align}
\left\| \bm{\nu}_{i+1} - \bm{\lambda}^* \right\| = O \left( \| \bm{z}_{i+1} - \bm{x}^* \| \right) \label{NPASA8.8.1}
\end{align}
which combined with (\ref{NPASA8.8}) yields
\begin{align}
\| \bm{z}_{i+1} - \bm{x}^* \| ^2
&\geq c_1 \left( \| \nabla \mathcal{L} \left( \bm{z}_{i+1}, \bm{\nu}_{i+1}, \bm{\mu}^* \right) \|^2 + \| \bm{h} (\bm{z}_{i+1}) \|^2 \right). \label{NPASA8.8.2}
\end{align}
%%%%%%%%%%%%%%%%%%%%%%%%%%%%%%%%%%%%%%%%%
\iffalse
Now recall the function $\bm{\Phi} : \mathbb{R}^m \times \mathbb{R}^m \to \mathbb{R}^m$ where the $j$th component of $\bm{\Phi}$ is defined to be $\Phi_j (\bm{a}, \bm{b}) = \min \{ a_j, b_j \}$. As $(\bm{x}^*, \bm{\lambda}^*, \bm{\mu}^*)$ is a KKT point for problem (\ref{prob:main-nlp}), we have that $\mu_j^* r_j (\bm{x}^*) = 0$, for $1 \leq j \leq m$. Now fix $j \in \{ 1, 2, \ldots, m \}$. Suppose $\mu_j^* = 0$. As $\bm{r} (\bm{z}_{i+1}) \leq \bm{0}$, we have
\begin{align}
\Phi_j (\bm{\mu}^*, -\bm{r} (\bm{z}_{i+1}))^2
&= \min \{ \mu_j^*, -r_j (\bm{z}_{i+1}) \}^2
= \min \{ 0, -r_j (\bm{z}_{i+1}) \}^2
= 0. \label{NPASA8.8.3}
\end{align}
On the other hand, suppose $r_j (\bm{x}^*) = 0$. Then
\begin{align}
\Phi_j (\bm{\mu}^*, -\bm{r} (\bm{z}_{i+1}))^2
&= \min \{ \mu_j^*, -r_j (\bm{z}_{i+1}) \}^2 \nonumber \\
&\leq \left( r_j (\bm{z}_{i+1}) \right)^2 \nonumber \\
&= \left( r_j (\bm{z}_{i+1}) - r_j (\bm{x}^*) \right)^2 \nonumber \\
&\leq c \| \bm{z}_{i+1} - \bm{x}^* \|^2. \label{NPASA8.8.4}
\end{align}
Combining (\ref{NPASA8.8.3}) and (\ref{NPASA8.8.4}) we have that there exists a constant $c$ such that
\begin{align}
\| \bm{\Phi} (\bm{\mu}^*, -\bm{r} (\bm{z}_{i+1})) \|^2
&\leq c \| \bm{z}_{i+1} - \bm{x}^* \|^2. \label{NPASA8.8.5}
\end{align}
\fi
%%%%%%%%%%%%%%%%%%%%%%%%%%%%%%%%%%%%%%%%%
By an argument similar to the one establishing (\ref{NPASA8.7.5}), we can show that
\begin{align}
c_1 \| \bm{\Phi} (\bm{\mu}^*, -\bm{r} (\bm{z}_{i+1})) \|^2
&\leq \| \bm{z}_{i+1} - \bm{x}^* \|^2. \label{NPASA8.8.5}
\end{align}
%Recalling that $E_{m,1} \left( \bm{x}, \bm{\lambda}, \bm{\mu} \right) = \| \nabla \mathcal{L} \left( \bm{x}, \bm{\lambda}, \bm{\mu} \right) \|^2 + \| \bm{\Phi} (\bm{\mu}, -\bm{r} (\bm{x})) \|^2$ 
It now follows from (\ref{NPASA8.8.2}) and (\ref{NPASA8.8.5}) that 
\begin{align}
c_1 E_{m,1} \left( \bm{z}_{i+1}, \bm{\nu}_{i+1}, \bm{\mu}^* \right) 
\leq \| \bm{z}_{i+1} - \bm{x}^* \|^2. \label{NPASA8.8.6}
\end{align}
As $\bm{\mu}^* \geq \bm{0}$, using the definition of $\bm{\eta}_{i+1}'$ in (\ref{scheme:MS2}) with (\ref{NPASA8.8.6}) we now have
\begin{align}
c_1 E_{m,1} \left( \bm{z}_{i+1}, \bm{\nu}_{i+1}, \bm{\eta}_{i+1}' \right) 
&\leq c_1 E_{m,1} \left( \bm{z}_{i+1}, \bm{\nu}_{i+1}, \bm{\mu}^* \right)
\leq \| \bm{z}_{i+1} - \bm{x}^* \|^2. \label{NPASA8.8.8}
\end{align}
Finally, recalling that $E_1 \left( \bm{z}, \bm{\nu}, \bm{\eta} \right)^2 = E_{m,1} \left( \bm{z}, \bm{\nu}, \bm{\eta} \right) + \| \bm{h} (\bm{z}) \|^2$ it follows from (\ref{NPASA8.8.2}) and (\ref{NPASA8.8.8}) that
\begin{align}
E_1 \left( \bm{z}_{i+1}, \bm{\nu}_{i+1}, \bm{\eta}_{i+1}' \right) \leq c \| \bm{z}_{i+1} - \bm{x}^* \|, \label{NPASA8.8.9}
\end{align}
which concludes the proof of (\ref{ConvergenceResult.5}).
\end{proof}

%%%%%%%%%%%%%%%%%%%%%%%%%%%%%%%%%%%%%%%%%%%%%%%%%%%%%%%%%%%%%%%%%%%%%%%%%%%%%%%
%%% ------------------------ CMS: LOCAL CONVERGENCE ----------------------- %%%
%%%%%%%%%%%%%%%%%%%%%%%%%%%%%%%%%%%%%%%%%%%%%%%%%%%%%%%%%%%%%%%%%%%%%%%%%%%%%%%
\subsection{Local Convergence of Phase Two of NPASA} \label{subsec:npasa-phase-two-local-conv}
Now that we have established local convergence results for the constraint and multiplier step problems, we are ready to provide a local convergence result for phase two of NPASA. Phase two of NPASA is comprised of the LS algorithm, Algorithm~\ref{alg:cms}, and branching criterion. 
%TODO
%We note that (\ref{cor:PhaseTwoConvergence.0}) and its proof are similar to Corollary 5.1 and its proof from \cite{Hager1993}.

\begin{corollary} \label{cor:PhaseTwoConvergence}
Suppose $f, \bm{h} \in \mathcal{C}^2$ and that either assumptions (\textbf{LICQ}), (\textbf{SCS}), and (\textbf{SOSC}) are satisfied at $(\bm{x}^*, \bm{\lambda}^*, \bm{\mu}^*)$ or assumptions (\textbf{LICQ}) and (\textbf{SSOSC}) are satisfied at $(\bm{x}^*, \bm{\lambda}^*, \bm{\mu}^*)$. Then there exists a neighborhood $\mathcal{N}_1$ of $\bm{x}^*$ and a constant $c$ such that, for every $\bm{x}_0 \in \mathcal{N}_1 \cap \Omega$ and every $k \geq 0$, the iterates of phase two of NPASA satisfy
\iffalse
\begin{itemize}
\item[(\ref{cor:PhaseTwoConvergence}.0)] $\| \bm{x}_{k+1} - \bm{x}^* \| \leq c \| \bm{x}_{k} - \bm{x}^* \|^2$,
\item[(\ref{cor:PhaseTwoConvergence}.1)] $E_1 \left( \bm{x}_{k+1}, \bm{\lambda}_{k+1}, \bm{\mu}_{k+1} \right) \leq c E_1 \left( \bm{x}_k, \bm{\lambda}_k, \bm{\mu}_k \right)^2$.
\end{itemize}
\fi
\begin{align}
    &\| \bm{x}_{k+1} - \bm{x}^* \| \leq c \| \bm{x}_{k} - \bm{x}^* \|^2, \label{cor:PhaseTwoConvergence.0} \\
    &E_1 \left( \bm{x}_{k+1}, \bm{\lambda}_{k+1}, \bm{\mu}_{k+1} \right) \leq c E_1 \left( \bm{x}_k, \bm{\lambda}_k, \bm{\mu}_k \right)^2. \label{cor:PhaseTwoConvergence.1}
\end{align}
\end{corollary}

\begin{proof}
The proof of (\ref{cor:PhaseTwoConvergence.0}) follows from applying Theorems~\ref{NewtonYSchemeThm} and~\ref{ConvergenceResult}. Applying (\ref{ConvergenceResult.3}), (\ref{NewtonYSchemeThm.3}), and (\ref{NewtonYSchemeThm.1}) together with the triangle and Cauchy inequalities yields %Observe that
\begin{align}
\| \bm{x}_{k+1} - \bm{x}^* \| 
&\leq c \| \bm{w} - \bm{x}^* \|^2 + c \| \bm{h} (\bm{w}) \| \\ %\tag{Applying \ref{ConvergenceResult.3}} \\
&\leq c \| \bm{w} - \bm{x}^* \|^2 + c \| \bm{h} (\bm{x}_{k}) \|^2 \\ %\tag{Applying \ref{NewtonYSchemeThm.3}} \\
&\leq c \left( \| \bm{w} - \bm{x}_k \| + \| \bm{x}_k - \bm{x}^* \| \right)^2 + c \| \bm{h} (\bm{x}_{k}) \|^2 \\ %\tag{Using triangle inequality} \\
&\leq c \left( \| \bm{h} (\bm{x}_k) \| + \| \bm{x}_k - \bm{x}^* \| \right)^2 + c \| \bm{h} (\bm{x}_{k}) \|^2 \\ %\tag{Applying \ref{NewtonYSchemeThm.1}} \\
&\leq \frac{5c}{2} \| \bm{h} (\bm{x}_k) \|^2 + \frac{3c}{2} \| \bm{x}_k - \bm{x}^* \|^2. %\tag{Using Cauchy inequality}
\end{align}
By performing a first-order Taylor expansion of $\bm{h} (\bm{x})$ and using the fact that $\bm{h} (\bm{x}^*) = \bm{0}$, there exists some $\bm{\xi}_k$ between $\bm{x}_k$ and $\bm{x}^*$ such that $\bm{h} (\bm{x}_k) = \nabla \bm{h} (\bm{\xi}_k) (\bm{x}_k - \bm{x}^*)$. Hence, $\| \bm{h} (\bm{x}_k) \| \leq c \| \bm{x}_k - \bm{x}^* \|$. Making this substitution into our previous inequality yields
\begin{align}
\| \bm{x}_{k+1} - \bm{x}^* \| 
&\leq \frac{5c^2}{2} \| \bm{x}_k - \bm{x}^* \|^2 + \frac{3c}{2} \| \bm{x}_k - \bm{x}^* \|^2 
= C \| \bm{x}_k - \bm{x}^* \|^2,
\end{align}
which proves (\ref{cor:PhaseTwoConvergence.0}).

To prove (\ref{cor:PhaseTwoConvergence.1}), we use (\ref{ConvergenceResult.5}), (\ref{cor:PhaseTwoConvergence.0}), and Theorem~\ref{thm:E1-error-bound}. Observe that
\begin{align}
E_1 \left( \bm{x}_{k+1}, \bm{\lambda}_{k+1}, \bm{\mu}_{k+1} \right) 
\leq C \| \bm{x}_{k+1} - \bm{x}^* \| %\tag{Applying \ref{ConvergenceResult.5}} \\
\leq C \| \bm{x}_k - \bm{x}^* \|^2 %\tag{Applying \ref{cor:PhaseTwoConvergence.0}} \\
\leq C E_1 \left( \bm{x}_k, \bm{\lambda}_k, \bm{\mu}_k \right)^2, %\tag{Applying Theorem~\ref{thm:E1-error-bound}}
\end{align}
which concludes the proof.
\end{proof}
\section{Convergence Analysis of NPASA} \label{section:npasa-global-convergence}
Now that we have established local convergence of phase two of NPASA we are ready to prove the local convergence result for NPASA. The statement of the main result requires two other assumptions that arise during the global convergence analysis of NPASA. These assumptions are considered in more detail in the phase one analysis of NPASA in the companion paper \cite{diffenderfer2020global} but for the sake of completeness we include these assumptions here with minimal details. First, note that these assumptions make use of the augmented Lagrangian function given by 
\begin{align}
\mathcal{L}_q (\bm{x}, \bm{\nu}) = f (\bm{x}) + \bm{\nu}^{\intercal} \bm{h}(\bm{x}) + q \| \bm{h}(\bm{x}) \|^2, \label{def:augmented-Lagrangian}
\end{align}
where $q \in \mathbb{R}$ is a penalty parameter. Additionally, given $\bm{u} \in \mathbb{R}^n$, $\bm{\nu} \in \mathbb{R}^s$, and $q$ we define the level set $\mathcal{S}(\bm{u}, \bm{\nu}, q) := \{ \bm{x} \in \Omega : \mathcal{L}_q (\bm{x}, \bm{\nu}) \leq \mathcal{L}_q (\bm{u}, \bm{\nu}) \}$. Lastly, we define $\mathcal{D}(\bm{u}, \bm{\nu}, q)$ to be the set of search directions generated by phase one of the algorithm PASA \cite{HagerActive} when solving the problem 
\begin{align}
\begin{array}{cc}
\displaystyle \min_{\bm{x}} & \displaystyle \mathcal{L}_q (\bm{x}, \bm{\nu}) \\
\text{s.t.} & \bm{x} \in \Omega
\end{array} \label{prob:aug-lag-intro}
\end{align}
starting at the initial guess $\bm{x} = \bm{u}$. Note that the definition of the search directions generated in phase one of PASA can be found in Algorithm 1 in Section 2 of \cite{HagerActive} and that problem (\ref{prob:aug-lag-intro}) is found in line 3 of the GS pseudocode, Algorithm~\ref{alg:gs}, in Section~\ref{section:npasa}. We are now ready to state the assumptions: 
\begin{enumerate}
\item[(G1)] Given $\bm{u} \in \mathbb{R}^n$, $\bm{\nu} \in \mathbb{R}^s$ and $q$, $\mathcal{L}_q (\bm{x}, \bm{\nu})$ is bounded from below on the level set $\mathcal{S} (\bm{u}, \bm{\nu})$ and $d_{\max} = \sup_{\bm{d} \in \mathcal{D}(\bm{u}, \bm{\nu}, q)} \| \bm{d} \| < \infty$,
\item[(G2)] Given $\bm{u} \in \mathbb{R}^n$ and $\bm{\nu} \in \mathbb{R}^s$, if $\overline{\mathcal{S}}(\bm{u}, \bm{\nu})$ is the collection of $\bm{x} \in \Omega$ whose distance to $\mathcal{S}(\bm{u}, \bm{\nu})$ is at most $d_{\max}$, then $\nabla_x \mathcal{L}_q (\bm{x}, \bm{\nu})$ is Lipschitz continuous on $\overline{\mathcal{S}}(\bm{u}, \bm{\nu})$.
\end{enumerate}

We now state the main result which establishes local quadratic convergence for the primal iterates and global error estimator for NPASA. Furthermore, this result establishes that only phase two of NPASA is executed after finitely many iterations.
\begin{theorem}[NPASA Local Convergence Theorem] \label{thm:npasa-local-conv}
Suppose that $\Omega$ is compact and that $f, \bm{h} \in \mathcal{C}^2 (\Omega)$. Suppose that NPASA (Algorithm~\ref{alg:npasa}) with $\varepsilon = 0$ generates a sequence $\{ (\bm{x}_k, \bm{\lambda}_k, \bm{\mu}_k) \}_{k=0}^{\infty}$ with $\bm{x}_k \in \Omega$ and let $\mathcal{S}_j$ be the set of indices such that if $k \in \mathcal{S}_j$ then $\bm{x}_k$ is generated in phase $j$ of NPASA, for $j \in \{ 1, 2 \}$. Suppose that the following assumptions hold:
\begin{itemize}
\item[\emph{(H1)}] For every $k \in \mathcal{S}_1$, assumptions (G1) and (G2) hold at $(\bm{x}_{k-1}, \bm{\nu}_{k-1}, q_k)$, where we have $\bm{\nu}_{k-1} = Proj_{[-\bar{\lambda},\bar{\lambda}]} (\bm{\lambda}_{k-1})$ and $\bar{\lambda} > 0$ is a scalar parameter.
\item[\emph{(H$2^*$)}] If $\bm{x}^*$ is a subsequential limit point of $\{ \bm{x}_k \}_{k=0}^{\infty}$ then (\textbf{LICQ}) holds at $\bm{x}^*$ and there exist KKT multipliers $\bm{\lambda}^*$ and $\bm{\mu}^*$ such that either (\textbf{SCS}) and (\textbf{SOSC}) hold at $(\bm{x}^*, \bm{\lambda}^*, \bm{\mu}^*)$ or (\textbf{SSOSC}) holds at $(\bm{x}^*, \bm{\lambda}^*, \bm{\mu}^*)$. 
\end{itemize}
Then there exists an integer $M$ such that if $k > M$ then $k \in \mathcal{S}_2$. Additionally, there exists a constant $c$ such that the following hold for all $k \geq M$:
\iffalse
\begin{itemize}
\item[(\ref{thm:npasa-local-conv}.0)] $\| \bm{x}_{k+1} - \bm{x}^* \| \leq c \| \bm{x}_{k} - \bm{x}^* \|^2$,
\item[(\ref{thm:npasa-local-conv}.1)] $E_1 \left( \bm{x}_{k+1}, \bm{\lambda}_{k+1}, \bm{\mu}_{k+1} \right) \leq c E_1 \left( \bm{x}_k, \bm{\lambda}_k, \bm{\mu}_k \right)^2$.
\end{itemize}
\fi
\begin{align}
&\| \bm{x}_{k+1} - \bm{x}^* \| \leq c \| \bm{x}_{k} - \bm{x}^* \|^2, \label{thm:npasa-local-conv.0} \\
&E_1 \left( \bm{x}_{k+1}, \bm{\lambda}_{k+1}, \bm{\mu}_{k+1} \right) \leq c E_1 \left( \bm{x}_k, \bm{\lambda}_k, \bm{\mu}_k \right)^2. \label{thm:npasa-local-conv.1}
\end{align}
\end{theorem}

\begin{proof}
Let $\mathcal{N}$ and $C$ be the neighborhood and constant given in Corollary~\ref{cor:PhaseTwoConvergence}, respectively. As the hypotheses of Theorem~1.1 from the companion paper \cite{diffenderfer2020global} are satisfied, there exists an integer $M$ such that $(\bm{x}_M, \bm{\lambda}_M, \bm{\mu}_M) \in \mathcal{N}$ and
\begin{align}
E_1 \left( \bm{x}_M, \bm{\lambda}_M, \bm{\mu}_M \right) \leq \frac{\theta}{C}. %\theta \min \left\{ 1, \frac{1}{C} \right\}. 
\label{eq:npasa-local-conv.1}
\end{align}
Suppose $M \in \mathcal{S}_1$. As the branching criterion for phase one are always satisfied under our hypothesis, we then branch to phase two of NPASA. If $M \in \mathcal{S}_2$, then we remain in branch two by definition of NPASA in Algorithm~\ref{alg:npasa}. Hence, after generating the point $(\bm{x}_M, \bm{\lambda}_M, \bm{\mu}_M)$ we are in phase two of NPASA. Now starting at the point $(\bm{x}_M, \bm{\lambda}_M, \bm{\mu}_M)$, phase two of NPASA generates an iterate which we will denote by $( \bm{x}_M', \bm{\lambda}_M', \bm{\mu}_M')$. Applying (\ref{cor:PhaseTwoConvergence.1}) yields that
\begin{align}
E_1 \left( \bm{x}_M', \bm{\lambda}_M', \bm{\mu}_M' \right)
&\leq C E_1 \left( \bm{x}_M, \bm{\lambda}_M, \bm{\mu}_M \right)^2. \label{eq:npasa-local-conv.2}
\end{align}
Now combining (\ref{eq:npasa-local-conv.1}) and (\ref{eq:npasa-local-conv.2}) it follows that
\begin{align}
E_1 \left( \bm{x}_M', \bm{\lambda}_M', \bm{\mu}_M' \right)
&\leq C \left( \frac{\theta}{C} \right) E_1 \left( \bm{x}_M, \bm{\lambda}_M, \bm{\mu}_M \right)
\leq \theta E_1 \left( \bm{x}_M, \bm{\lambda}_M, \bm{\mu}_M \right). \label{eq:npasa-local-conv.3}
\end{align}
%%%%%%%%%%%%%%%%%%%%% (Old; More complicated than necessary)
\iffalse
\begin{align}
E_1 \left( \bm{x}_M', \bm{\lambda}_M', \bm{\mu}_M' \right)
&\leq C \theta \min \left\{ 1, \frac{1}{C} \right\} E_1 \left( \bm{x}_M, \bm{\lambda}_M, \bm{\mu}_M \right)
\leq \theta E_1 \left( \bm{x}_M, \bm{\lambda}_M, \bm{\mu}_M \right). \label{eq:npasa-local-conv.3}
\end{align}
\fi
%%%%%%%%%%%%%%%%%%%%%%%%%%%%%%%%%%%%%%%%%%%%%%%%%%%
From (\ref{eq:npasa-local-conv.3}) it follows that we set $( \bm{x}_{M+1}, \bm{\lambda}_{M+1}, \bm{\mu}_{M+1} ) \gets ( \bm{x}_M', \bm{\lambda}_M', \bm{\mu}_M')$ so that $(M+1) \in \mathcal{S}_2$ and NPASA does not branch to phase one. Additionally, as $\theta \in (0,1)$ we note that (\ref{eq:npasa-local-conv.1}) and (\ref{eq:npasa-local-conv.2}) yield
\begin{align}
E_1 \left( \bm{x}_{M+1}, \bm{\lambda}_{M+1}, \bm{\mu}_{M+1} \right)
&\leq C \left( \frac{\theta}{C} \right)^2
= \frac{\theta^2}{C}
< \frac{\theta}{C}. \label{eq:npasa-local-conv.4}
\end{align}
Now it follows from (\ref{eq:npasa-local-conv.3}) and (\ref{eq:npasa-local-conv.4}) that $k \in \mathcal{S}_2$ for every $k > M$. Finally, (\ref{thm:npasa-local-conv.0}) and (\ref{thm:npasa-local-conv.1}) follow from (\ref{cor:PhaseTwoConvergence.0}) and (\ref{cor:PhaseTwoConvergence.1}), respectively.
\end{proof}

\section{Conclusion} \label{section:conclusion}
In this paper, we provided local convergence analysis for NPASA which supports the design of NPASA originally presented and motivated in the companion paper \cite{diffenderfer2020global}. Specifically, we proved that NPASA only exectues phase two after finitely many iterations and that the primal iterates and global error estimator for NPASA achieve a quadratic convergence rate in a neighborhood of a minimizer under certain sets of assumptions. Based on the sets of assumptions, this local convergence result holds for nondegenerate and degenerate minimizers of problem (\ref{prob:main-nlp}). To continue this research we plan to implement the NPASA algorithm in order to compare this method to other algorithms designed to solve nonlinear programs.

This work was performed, in part,  under the auspices of the U.S. Department of Energy by Lawrence Livermore National Laboratory under Contract DE-AC52-07NA27344.
%\input{acknowledgments}
% Add bibliography before appendix
\bibliographystyle{siamplain}
\bibliography{paper}
% Add appendix
\clearpage
\appendix
\section{Constrained Optimization Definitions and Results} \label{appendixa}
In this appendix, we highlight key definitions and results from the constrained optimization theory that were used as assumptions in our analysis of NPASA. Statements of these results are included here for completeness but we note that more details on these theorems and proofs can be found by referencing \cite{NandW, bertsekas1995nonlinear}. As the hypotheses of many of these results are used as assumptions in establishing global and local convergence properties for NPASA, these results are only referenced in Section~\ref{subsec:notation} where we provide simplified abbreviations for the hypotheses of these theorems. In this appendix, given $i \in \mathbb{N}$ note that we will write $[i]$ to denote the set of integers $\{ 1, 2, \ldots, i \}$.
\iffalse
\begin{definition} \label{def:li}
A vector $\bm{x}$ is said to satisfy the {\bfseries linear independence condition} for problem (\ref{prob:main-nlp}) if the matrix {\footnotesize $\begin{bmatrix} \nabla \bm{h} (\bm{x}) \\ \bm{A} \end{bmatrix}$} is of full row rank.
\end{definition}
\fi
\begin{definition} \label{def:licq}
A vector $\bm{x}$ is said to satisfy the {\bfseries linear independence constraint qualification condition} for problem (\ref{prob:main-nlp}) if the matrix {\footnotesize $\displaystyle \begin{bmatrix} \nabla \bm{h} (\bm{x}) \\ \bm{A}_{\mathcal{A}(\bm{x})} \end{bmatrix}$} is of full row rank.
\end{definition}

As in Section~\ref{subsec:notation}, note that if $\bm{x}$ satisfies the linear independence constraint qualification condition then we abbreviate this by writing (\textbf{LICQ}) holds at $\bm{x}$. We now state the first order optimality or KKT conditions.

\begin{theorem}[Karush-Kuhn-Tucker Conditions \cite{NandW}] \label{thm:kkt}
Suppose $\bm{x}^*$ is a local solution of problem (\ref{prob:main-nlp}), that $f, \bm{h} \in \mathcal{C}^1$, and (\textbf{LICQ}) holds at $\bm{x}^*$. Then there exist KKT multipliers $\bm{\lambda}^*$ and $\bm{\mu}^*$ such that the following conditions hold at $(\bm{x}^*, \bm{\lambda}^*, \bm{\mu}^*)$:
\begin{enumerate}[leftmargin=2\parindent,align=left,labelwidth=\parindent,labelsep=7pt]
	\item[(\textbf{KKT.1})] Gradient of Lagrangian equals zero: $\nabla_x \mathcal{L} (\bm{x}^*, \bm{\lambda}^*, \bm{\mu}^*)^{\intercal} = \bm{0}$
	\item[(\textbf{KKT.2})] Satisfies equality constraints: $\bm{h} (\bm{x}^*) = \bm{0}$
	\item[(\textbf{KKT.3})] Satisfies inequality constraints: $\bm{r} (\bm{x}^*) \leq \bm{0}$
	\item[(\textbf{KKT.4})] Nonnegativity of inequality multipliers: $\bm{\mu}^* \geq \bm{0}$
	\item[(\textbf{KKT.5})] Complementary slackness: $r_i (\bm{x}^*) \mu_i^* = 0$ for $i \in [m]$.
\end{enumerate}
\end{theorem}

Before stating second order optimality conditions we first provide several definitions required for the statements of these theorems.

\begin{definition} \label{def:scs}
A point $(\bm{x}, \bm{\lambda}, \bm{\mu})$ is said to satisfy {\bfseries strict complementary slackness} for problem (\ref{prob:main-nlp}) if it satisfies (\textbf{KKT.1}) -- (\textbf{KKT.5}) and exactly one of $r_i (\bm{x})$ and $\mu_i$ is zero for each $i \in [m]$.
\end{definition}

\begin{definition} \label{def:critical-cone}
A point $\bm{w}$ is said to be in the {\bfseries critical cone} at the point $(\bm{x}^*, \bm{\lambda}^*, \bm{\mu}^*)$, denoted $\displaystyle \bm{w} \in \mathcal{C} (\bm{x}^*, \bm{\lambda}^*, \bm{\mu}^*)$, if and only if 
\begin{align}
\left\{
	\begin{array}{lll}
	\nabla h_i (\bm{x}^*)^{\intercal} \bm{w} = 0 & : & \text{for all} \ i \in [s] \vspace{0.8mm} \\
	\nabla r_i (\bm{x}^*)^{\intercal} \bm{w} = 0 & : & \text{for all} \ i \in \mathcal{A}(\bm{x}^*) \cap [m] \ \text{with} \ \mu_i^* > 0 \vspace{0.8mm} \\
    \nabla r_i (\bm{x}^*)^{\intercal} \bm{w} \leq 0 & : & \text{for all} \ i \in \mathcal{A}(\bm{x}^*) \cap [m] \ \text{with} \ \mu_i^* = 0
	\end{array}
\right.
\end{align}
%\{ \bm{w} \in \mathcal{F} (\bm{x}^*) \vert \nabla \bm{r}_i (\bm{x}^*)
\end{definition}

\begin{definition} \label{def:mod-critical-cone}
A point $\bm{w}$ is said to be in the {\bfseries modified critical cone} at the point $(\bm{x}^*, \bm{\lambda}^*, \bm{\mu}^*)$, denoted $\displaystyle \bm{w} \in \mathcal{C}' (\bm{x}^*, \bm{\lambda}^*, \bm{\mu}^*)$, if and only if 
\begin{align}
\left\{
	\begin{array}{lll}
	\nabla h_i (\bm{x}^*)^{\intercal} \bm{w} = 0 & : & \text{for all} \ i \in [s] \vspace{0.8mm} \\
	\nabla r_i (\bm{x}^*)^{\intercal} \bm{w} = 0 & : & \text{for all} \ i \in \mathcal{A}(\bm{x}^*) \cap [m] \ \text{with} \ \mu_i^* > 0
	\end{array}
	\right.
\end{align}
%\{ \bm{w} \in \mathcal{F} (\bm{x}^*) \vert \nabla \bm{r}_i (\bm{x}^*)
\end{definition}

Based on the definitions of the critical cone and the modified critical cone, note that $\mathcal{C} (\bm{x}^*, \bm{\lambda}^*, \bm{\mu}^*) \subset \mathcal{C}' (\bm{x}^*, \bm{\lambda}^*, \bm{\mu}^*)$. We now state two theorems that provide sufficient second order optimality conditions for problem (\ref{prob:main-nlp}).

\begin{theorem}[Second-Order Sufficient Optimality Conditions \cite{NandW}] \label{thm:sosc}
Suppose $f, \bm{h} \in \mathcal{C}^2(\Omega)$ and that for some feasible point $\bm{x}^*$ for problem (\ref{prob:main-nlp}) there are KKT multipliers $\bm{\lambda}^*$ and $\bm{\mu}^*$ such that (\textbf{KKT.1}) -- (\textbf{KKT.5}) hold for $(\bm{x}^*, \bm{\lambda}^*, \bm{\mu}^*)$. Additionally, suppose that 
\begin{align}
\bm{w}^T \nabla_{xx}^2 \mathcal{L} (\bm{x}^*, \bm{\lambda}^*, \bm{\mu}^*) \bm{w} > 0, \ \ \text{for all} \ \ \bm{w} \in \mathcal{C} (\bm{x}^*, \bm{\lambda}^*, \bm{\mu}^*) \setminus \{ \bm{0} \}
\end{align}
Then $\bm{x}^*$ is a strict local solution for problem (\ref{prob:main-nlp}).
\end{theorem}

\begin{theorem}[Strong Second-Order Sufficient Optimality Conditions \cite{Robinson1980}] \label{thm:ssosc}
Suppose $f, \bm{h} \in \mathcal{C}^2(\Omega)$ and that for some feasible point $\bm{x}^*$ for problem (\ref{prob:main-nlp}) there are KKT multipliers $\bm{\lambda}^*$ and $\bm{\mu}^*$ such that (\textbf{KKT.1}) -- (\textbf{KKT.5}) hold for $(\bm{x}^*, \bm{\lambda}^*, \bm{\mu}^*)$. Additionally, suppose that 
\begin{align}
\bm{w}^T \nabla_{xx}^2 \mathcal{L} (\bm{x}^*, \bm{\lambda}^*, \bm{\mu}^*) \bm{w} > 0, \ \ \text{for all} \ \ \bm{w} \in \mathcal{C}' (\bm{x}^*, \bm{\lambda}^*, \bm{\mu}^*) \setminus \{ \bm{0} \}
\end{align}
Then $\bm{x}^*$ is a strict local solution for problem (\ref{prob:main-nlp}).
\end{theorem}

\section{Constructing Inequality Multiplier in Global Step} \label{subsec:gs-inequality-multiplier}
%At this point, we have addressed nearly every step of phase one of NPASA. However, throughout the analysis in this section we have 
%Here, we outline how to construct the inequality multiplier $\bm{\mu} (\bm{x}_k, 1)$ required by the GS algorithm, Algorithm~\ref{alg:gs}, using values internally computed by PPROJ. 
We now discuss how to construct the inequality multiplier $\bm{\mu} (\bm{x}_k, 1)$ required by the GS algorithm, Algorithm~\ref{alg:gs}, by using information already computed internally by modified PASA (Algorithm~3.4 of the companion paper \cite{diffenderfer2020global}) when solving the minimization problem in Algorithm~\ref{alg:gs}.
%While the formula for $E_{m,0} (\bm{x}_k, \bm{\lambda}_k, \bm{\mu} (\bm{x}_k,1))$ provided in 
%(\ref{eq:aug-lag-mod.4.2}) 
%the companion paper \cite{diffenderfer2020global} proved useful in establishing convergence of the augmented Lagrangian problem phase one of NPASA and allows the user to compute $E_{m,0}$ without requiring the value of $\bm{\mu} (\bm{x}_k,1)$, to establish desirable global convergence properties for NPASA in Section~\ref{section:npasa-global-convergence} we need to compute $E_{m,1} (\bm{x}_k, \bm{\lambda}_k, \bm{\mu} (\bm{x}_k,1))$ at the end of phase one of NPASA. 
The vector $\bm{\mu} (\bm{x}_k, 1)$ is required to compute $E_1 (\bm{x}_k, \bm{\lambda}_k, \bm{\mu} (\bm{x}_k, 1))$ to update $e_k$ and check the branching and stopping criterion in NPASA. 

When checking the stopping criterion 
%for the original PASA algorithm \cite{HagerActive} or 
for modified PASA (Algorithm~3.4 of the companion paper \cite{diffenderfer2020global}) it is necessary to solve the problem %(\ref{prob:pasa-error}) with $\alpha = 1$, restated here as
\begin{align}
\begin{array}{cc}
\min & \displaystyle \frac{1}{2} \| \bm{x} - \bm{g} (\bm{x}) - \bm{y} \|^2 \vspace{2mm}\\
\text{s.t.} & \bm{y} \in \Omega'
\end{array} \label{prob:pasa-construct-mult}
\end{align}
where $\bm{x}$ is fixed and we let $\Omega'$ be of the general form $\Omega' = \{ \bm{y} : \bm{\ell} \leq \bm{A} \bm{y} \leq \bm{u}, \bm{lo} \leq \bm{y} \leq \bm{hi} \}$. 
In an implementation of NPASA, problem (\ref{prob:pasa-construct-mult}) can be solved using a dual technique with the algorithm PPROJ \cite{Hager2016}. From equation (1.5) in \cite{Hager2016} PPROJ computes a solution for the dual variable, which we will denote by $\bm{\pi}^*$, that is used to reconstruct the primal solution, $\bm{y} (\bm{\pi}^*)$. Letting $\bm{a}_i$ denote the $i$th column of $\bm{A}$ and $\bm{a}_j^{\intercal}$ denote the $j$th row of $\bm{A}$, using equation (1.5) in \cite{Hager2016} we have that this primal solution satisfies
\begin{align}
y_i (\bm{\pi}^*) = \left\{
	\begin{array}{lll}
	lo_i & : & x_i - g_i (\bm{x}) + \bm{a}_i^{\intercal} \bm{\pi}^* \leq lo_i \vspace{0.8mm} \\
	hi_i & : & x_i - g_i (\bm{x}) + \bm{a}_i^{\intercal} \bm{\pi}^* \geq hi_i \vspace{0.8mm} \\
    x_i - g_i (\bm{x}) + \bm{a}_i^{\intercal} \bm{\pi}^* & : & \text{otherwise}
	\end{array}
	\right. \label{eq:pasa-construct-mult.1}
\end{align}
and
\begin{align}
\bm{a}_j^{\intercal} \bm{y} (\bm{\pi}^*) = \left\{
	\begin{array}{lll}
	\ell_j & : & \pi_j^* > 0 \\
	u_j & : & \pi_j^* < 0
	\end{array}
	\right., \label{eq:pasa-construct-mult.2}
\end{align}
for $1 \leq i \leq n$ and $1 \leq j \leq m$. Now consider the Lagrangian for problem (\ref{prob:pasa-construct-mult}) given by
\begin{align}
L (\bm{y}, \bm{\gamma}_1, \bm{\gamma}_2, \bm{\upsilon}_1, \bm{\upsilon}_2) 
&= \frac{1}{2} \| \bm{x} - \bm{g} (\bm{x}) - \bm{y} \|^2 + \bm{\gamma}_1^{\intercal} (\bm{\ell} - \bm{A} \bm{y}) + \bm{\gamma}_2^{\intercal} (\bm{A} \bm{y} - \bm{u}) \label{eq:pasa-construct-mult.3} \\
&\quad + \bm{\upsilon}_1^{\intercal} (\bm{lo} - \bm{y}) + \bm{\upsilon}_2^{\intercal} ( \bm{y} - \bm{hi}). \nonumber
\end{align}
It follows that
\begin{align}
\nabla_y L (\bm{y}, \bm{\gamma}_1, \bm{\gamma}_2, \bm{\upsilon}_1, \bm{\upsilon}_2)^{\intercal} 
&= \bm{y} - (\bm{x} - \bm{g} (\bm{x}))  - \bm{A}^{\intercal} \bm{\gamma}_1 + \bm{A}^{\intercal} \bm{\gamma}_2 - \bm{\upsilon}_1 + \bm{\upsilon}_2. \label{eq:pasa-construct-mult.4}
\end{align}
%Letting $\bm{y}^* = \bm{y} (\bm{\pi}^*)$, we proceed by partitioning $\bm{y}^*$ into three parts, denoted $\bm{y}_{\mathcal{S}_1}^*$, $\bm{y}_{\mathcal{S}_2}^*$, and $\bm{y}_{\mathcal{S}_3}^*$, where the index sets used to construct the partition of $\bm{y}^*$ are defined by $\mathcal{S}_1 = \{ i \in \mathbb{N} : \bm{y}_i^* = lo_i \}$, $\mathcal{S}_2 = \{ i \in \mathbb{N} : \bm{y}_i^* = hi_i \}$, and $\mathcal{S}_3 = \{ i \in \mathbb{N} : lo_i < \bm{y}_i^* < hi_i \}$. 
Now let $\bm{y}^* = \bm{y} (\bm{\pi}^*)$ and define the index sets $\mathcal{S}_1 = \{ i \in \mathbb{N} : \bm{y}_i^* = lo_i \}$, $\mathcal{S}_2 = \{ i \in \mathbb{N} : \bm{y}_i^* = hi_i \}$, and $\mathcal{S}_3 = \{ i \in \mathbb{N} : lo_i < \bm{y}_i^* < hi_i \}$. Using $\bm{y}^*$, $\bm{\pi}^*$, and the sets $\mathcal{S}_j$, for $j \in \{ 1, 2, 3 \}$, we work to construct multiplier vectors $\bm{\gamma}_1^*$, $\bm{\gamma}_2^*$, $\bm{\upsilon}_1^*$, and $\bm{\upsilon}_2^*$ that satisfy the KKT conditions for problem (\ref{prob:pasa-construct-mult}) given by:
\iffalse
\begin{enumerate}[leftmargin=2\parindent,align=left,labelwidth=\parindent,labelsep=7pt]
\item[(\ref{prob:pasa-construct-mult}.1)] Gradient of Lagrangian equals zero: $\nabla_y L (\bm{y}^*, \bm{\gamma}_1^*, \bm{\gamma}_2^*, \bm{\upsilon}_1^*, \bm{\upsilon}_2^*)^{\intercal} = \bm{0}$
\item[(\ref{prob:pasa-construct-mult}.2)] Satisfies problem constraints: $\bm{y}^* \in \Omega$
\item[(\ref{prob:pasa-construct-mult}.3)] Nonnegativity of inequality multipliers: $\bm{\gamma}_1^*, \bm{\gamma}_2^* \in \mathbb{R}_{\geq 0}^{m}$ and $\bm{\upsilon}_1^*, \bm{\upsilon}_2^* \in \mathbb{R}_{\geq 0}^{n}$
\item[(\ref{prob:pasa-construct-mult}.4)] Complementary slackness: $(\bm{\ell} - \bm{A} \bm{y}^*)^{\intercal} \bm{\gamma}_1^* = 0$, $(\bm{A} \bm{y}^* - \bm{u})^{\intercal} \bm{\gamma}_2^* = 0$, $(\bm{lo} - \bm{y}^*)^{\intercal} \bm{\upsilon}_1^* = 0$, and $(\bm{y}^* - \bm{hi})^{\intercal} \bm{\upsilon}_2^* = 0$.
\end{enumerate}
\fi
\begin{align}
&\text{Gradient of Lagrangian equals zero:} \ \nabla_y L (\bm{y}^*, \bm{\gamma}_1^*, \bm{\gamma}_2^*, \bm{\upsilon}_1^*, \bm{\upsilon}_2^*)^{\intercal} = \bm{0}, \label{prob:pasa-construct-mult.1} \\ 
&\text{Satisfies problem constraints:} \ \bm{y}^* \in \Omega, \label{prob:pasa-construct-mult.2} \\
&\text{Nonnegativity of inequality multipliers:} \ \bm{\gamma}_1^*, \bm{\gamma}_2^* \in \mathbb{R}_{\geq 0}^{m} \ \text{and} \ \bm{\upsilon}_1^*, \bm{\upsilon}_2^* \in \mathbb{R}_{\geq 0}^{n}, \label{prob:pasa-construct-mult.3} \\
&\text{Compl. slackness:} \ (\bm{\ell} - \bm{A} \bm{y}^*)^{\intercal} \bm{\gamma}_1^* = 0, (\bm{A} \bm{y}^* - \bm{u})^{\intercal} \bm{\gamma}_2^* = 0, (\bm{lo} - \bm{y}^*)^{\intercal} \bm{\upsilon}_1^* = 0, (\bm{y}^* - \bm{hi})^{\intercal} \bm{\upsilon}_2^* = 0. \label{prob:pasa-construct-mult.4}
\end{align}
Note that (\ref{prob:pasa-construct-mult.2}) is already satisfied. First, we define the components of vectors $\bm{\gamma}_1^*$ and $\bm{\gamma}_2^*$ by
\begin{align}
\gamma_{1,j}^* := \left\{
	\begin{array}{cll}
	\pi_j^* & : & \pi_j^* > 0 \\
	0 & : & \text{otherwise}
	\end{array}
	\right. \label{eq:pasa-construct-mult.5}
\end{align}
and
\begin{align}
\gamma_{2,j}^* := \left\{
	\begin{array}{cll}
	-\pi_j^* & : & \pi_j^* < 0 \\
	0 & : & \text{otherwise}
	\end{array}
	\right., \label{eq:pasa-construct-mult.6}
\end{align}
respectively, for $1 \leq j \leq m$. By definition, we have that $\bm{\gamma}_1^*, \bm{\gamma}_2^* \in \mathbb{R}_{\geq 0}^{m}$. Also, it follows from definition that
\begin{align}
\bm{\gamma}_1^* - \bm{\gamma}_2^* 
= \bm{\pi}^*. \label{eq:pasa-construct-mult.6.1}
\end{align}
Additionally, combining (\ref{eq:pasa-construct-mult.2}), (\ref{eq:pasa-construct-mult.5}), and (\ref{eq:pasa-construct-mult.6}) it follows that $(\bm{\ell} - \bm{A} \bm{y}^*)^{\intercal} \bm{\gamma}_1^* = 0$ and $(\bm{A} \bm{y}^* - \bm{u})^{\intercal} \bm{\gamma}_2^* = 0$. Next, we define vectors $\bm{\upsilon}_1^*$ and $\bm{\upsilon}_2^*$ componentwise by
\begin{align}
\upsilon_{1,i}^* := \left\{
	\begin{array}{lll}
	lo_i - (x_i - \bm{g}_i (\bm{x})) - \bm{a}_i^{\intercal} \bm{\pi}^* & : & i \in \mathcal{S}_1 \\
	0 & : & i \in \mathcal{S}_2 \cup \mathcal{S}_3
	\end{array}
	\right. \label{eq:pasa-construct-mult.7}
\end{align}
and
\begin{align}
\upsilon_{2,i}^* := \left\{
	\begin{array}{lll}
	-hi_i + x_i - \bm{g}_i (\bm{x}) + \bm{a}_i^{\intercal} \bm{\pi}^* & : & i \in \mathcal{S}_2 \\
	0 & : & i \in \mathcal{S}_1 \cup \mathcal{S}_3
	\end{array}
	\right., \label{eq:pasa-construct-mult.8}
\end{align}
respectively. By the definition of $\mathcal{S}_1$ and (\ref{eq:pasa-construct-mult.1}), we have that
\begin{align}
lo_i - (x_i - \bm{g}_i (\bm{x})) - \bm{a}_i^{\intercal} \bm{\pi}^* \geq 0, \quad \text{for all $i \in \mathcal{S}_1$.} \label{eq:pasa-construct-mult.9}
\end{align}
As $\upsilon_{1,i}^* = 0$ for $i \in \mathcal{S}_2 \cup \mathcal{S}_3$ it now follows that $\bm{\upsilon}_1^* \in \mathbb{R}_{\geq 0}^{n}$. Similarly, by the definition of $\mathcal{S}_2$ and (\ref{eq:pasa-construct-mult.1}) it follows that
\begin{align}
-hi_i + (x_i - \bm{g}_i (\bm{x})) + \bm{a}_i^{\intercal} \bm{\pi}^* \geq 0, \quad \text{for all $i \in \mathcal{S}_2$.} \label{eq:pasa-construct-mult.10}
\end{align}
As $\upsilon_{2,i}^* = 0$ for $i \in \mathcal{S}_1 \cup \mathcal{S}_3$ it now follows that $\bm{\upsilon}_2^* \in \mathbb{R}_{\geq 0}^{n}$. Hence, (\ref{prob:pasa-construct-mult.3}) holds. Next, combining (\ref{eq:pasa-construct-mult.2}), (\ref{eq:pasa-construct-mult.9}), and (\ref{eq:pasa-construct-mult.10}) it follows that $(\bm{lo} - \bm{y}^*)^{\intercal} \bm{\upsilon}_1^* = 0$ and $(\bm{y}^* - \bm{hi})^{\intercal} \bm{\upsilon}_2^* = 0$. Hence, (\ref{prob:pasa-construct-mult.4}) holds. Lastly, it remains to show that (\ref{prob:pasa-construct-mult.1}) holds. Combining (\ref{eq:pasa-construct-mult.4}) and (\ref{eq:pasa-construct-mult.6.1}) we have that
\begin{align}
\nabla_y L (\bm{y}^*, \bm{\gamma}_1^*, \bm{\gamma}_2^*, \bm{\upsilon}_1^*, \bm{\upsilon}_2^*)^{\intercal} 
&= \bm{y}^* - (\bm{x} - \bm{g} (\bm{x}))  - \bm{A}^{\intercal} \bm{\pi}^* - \bm{\upsilon}_1^* + \bm{\upsilon}_2^*. \label{eq:pasa-construct-mult.11}
\end{align}
To show that (\ref{prob:pasa-construct-mult.1}) holds, we need to show that
\begin{align}
\bm{y}_i^* - (x_i - \bm{g}_i (\bm{x}))  - \bm{a}_i^{\intercal} \bm{\pi}^* - \upsilon_{1,i}^* + \upsilon_{2,i}^*
&= 0. \label{eq:pasa-construct-mult.12}
\end{align}
for $1 \leq i \leq n$. Accordingly, fix $i \in \{1, 2, \ldots, n \}$. Suppose $i \in \mathcal{S}_1$. Then using (\ref{eq:pasa-construct-mult.1}), (\ref{eq:pasa-construct-mult.9}), and (\ref{eq:pasa-construct-mult.10}) we have that $\bm{y}_i^* = lo_i$, $\upsilon_{1,i}^* = lo_i - (x_i - \bm{g}_i (\bm{x})) - \bm{a}_i^{\intercal} \bm{\pi}^*$, and $\upsilon_{2,i}^* = 0$, respectively. Hence, (\ref{eq:pasa-construct-mult.12}) holds for $i \in \mathcal{S}_1$. Next suppose $i \in \mathcal{S}_2$. Then using (\ref{eq:pasa-construct-mult.1}), (\ref{eq:pasa-construct-mult.9}), and (\ref{eq:pasa-construct-mult.10}) we have that $\bm{y}_i^* = hi_i$, $\upsilon_{1,i}^* = 0$, and $\upsilon_{2,i}^* = -hi_i + x_i - \bm{g}_i (\bm{x}) + \bm{a}_i^{\intercal} \bm{\pi}^*$, respectively. Hence, (\ref{eq:pasa-construct-mult.12}) holds for $i \in \mathcal{S}_2$. Lastly, suppose $i \in \mathcal{S}_3$. Then using (\ref{eq:pasa-construct-mult.1}), (\ref{eq:pasa-construct-mult.9}), and (\ref{eq:pasa-construct-mult.10}) we have that $\bm{y}_i^* = x_i - \bm{g}_i (\bm{x}) + \bm{a}_i^{\intercal} \bm{\pi}^*$, $\upsilon_{1,i}^* = 0$, and $\upsilon_{2,i}^* = 0$, respectively. Hence, (\ref{eq:pasa-construct-mult.12}) holds for $i \in \mathcal{S}_3$. As $\mathcal{S}_1 \cup \mathcal{S}_2 \cup \mathcal{S}_3 = \{ 1, 2, \ldots, n \}$, we conclude that (\ref{prob:pasa-construct-mult.1}) holds. Thus, $(\bm{y}^*, \bm{\gamma}_1^*, \bm{\gamma}_2^*, \bm{\upsilon}_1^*, \bm{\upsilon}_2^*)$ is a KKT point for problem (\ref{prob:pasa-construct-mult}). As our goal was to construct the inequality multiplier $\bm{\mu} (\bm{x}, 1)$, we now note that $\bm{y}^* = \bm{y} (\bm{x}^*, 1)$. Hence, when the polyhedral constraints for the original problem are in the general format
\begin{align}
\bm{r} (\bm{x}) = \begin{bmatrix}
\bm{\ell} - \bm{A} \bm{x} \\
\bm{A} \bm{x} - \bm{u} \\
\bm{lo} - \bm{x} \\
\bm{x} - \bm{hi}
\end{bmatrix}
\end{align}
then it follows that we should construct $\bm{\mu} (\bm{x}^*, 1)$ by
\begin{align}
\bm{\mu} (\bm{x}^*, 1) := \begin{bmatrix}
\bm{\gamma}_1^* \\ 
\bm{\gamma}_2^* \\
\bm{\upsilon}_1^* \\
\bm{\upsilon}_2^*
\end{bmatrix}. \label{eq:pasa-construct-mult.13}
\end{align}
As $\bm{y}^*$ and $\bm{\pi}^*$ are available upon completion of PPROJ, we can always perform this construction. 
%Note that this construction only needs to be performed once after terminating modified PASA (Algorithm~3.4 in companion paper \cite{diffenderfer2020global}) in phase one of Algorithm~\ref{alg:npasa} since we only need $\bm{\mu} (\bm{x}_k, 1)$ to compute $E_1 (\bm{x}_k, \bm{\lambda}_k, \bm{\mu} (\bm{x}_k, 1))$ at the end of phase one.

\end{document}